\documentclass[a4paper,10pt]{article}
\usepackage{stmaryrd}
\usepackage{amsfonts}
\usepackage{bbm}
\usepackage{amscd}
\usepackage{mathrsfs}
\usepackage{latexsym,amssymb,amsmath,amscd,amscd,amsthm,amsxtra}
\usepackage[dvips]{graphicx}
\usepackage[utf8]{inputenc}
\usepackage[T1]{fontenc}
\usepackage{lmodern}
\usepackage{amssymb}
\usepackage[all]{xy}
\usepackage{nicefrac,mathtools,enumitem}
\usepackage{microtype}
\usepackage{CJK}
\usepackage{amssymb}
\usepackage{epstopdf}
\usepackage{graphicx}
\usepackage{graphics}
\usepackage[T1]{fontenc}
\usepackage[utf8]{inputenc}
\usepackage{authblk}
\usepackage{subfigure}
\usepackage{caption}
\usepackage{framed}
\usepackage{mathrsfs}

\numberwithin{equation}{section}
\newtheorem{thm}{Theorem}
\newtheorem{lem}{Lemma}
\newtheorem{cor}{Corollary}

\newtheorem{rmk}{Remark}

\newtheorem{ass}{Assumption}

\newcommand {\emptycomment}[1]{}

\newcommand{\be }{\begin{equation}}
\newcommand{\ee }{\end{equation}}

\newcommand{\pf}{\noindent{\bf Proof.}\ }

\newcommand{\G}{\mathbb G}

\newcommand{\frkd}{\mathfrak d}

\newcommand{\g}{\mathfrak g}

\newcommand{\nono}{\nonumber}
\newcommand{\noi}{\noindent}

\newcommand{\w}{\widetilde}

\newcommand{\f}{\frac}
\newcommand{\ttt}{\theta}
\newcommand{\aaa}{\alpha}

\newcommand{\nn}{\langle}
\newcommand{\mm}{\rangle}

\newcommand{\mbb}{\mathbb}
\newcommand{\eee}{\epsilon}
\newcommand{\ooo}{\omega}
\newcommand{\kkk}{\kappa}

\def\bea{\begin{eqnarray}}
\def\eea{\end{eqnarray}}
\def\be{\begin{equation}}
\def\ee{\end{equation}}
\def\blm{\begin{lem}}
\def\elm{\end{lem}}

\def\p{\mathcal{P}}
\def\b{\mathcal{B}}
\def\t{\mathcal{T}}
\def\c{\mathcal{C}}
\def\o{\mathcal{O}}

\def\G{G_{\mu,k}^{(i_k)}}
\def\pg{\p_{i_k}(x_k,U_{i_k}\G,\aaa_k)}
\def\pf{\p_{i_k}(x_k,\nabla f(x_k),\aaa_k)}
\def\pfm{\p_{i_k}(x_k,\nabla f_{\mu}(x_k),\aaa_k)}

\begin{document}
\title{
{Zeroth-Order Stochastic Block Coordinate Type Methods for Nonconvex Optimization\thanks{This work was partially supported by Research grants Council of the Hong Kong Special Administrative Region, China (CityU 11200717, CityU 11201518). Z. Yu and D. W. C. Ho are with the Department of Mathematics, City University of
Hong Kong, Kowloon, Hong Kong (e-mail: zhanyu2-c@my.cityu.edu.hk; madaniel@cityu.edu.hk).}}
 }\vspace{2mm}
\author{ Zhan Yu, Daniel W.C. Ho\emph{ }\emph{ }}
\date{}



\maketitle

\begin{abstract}
We study (constrained) nonconvex (composite) optimization problems where the decision variables vector can be split into blocks of variables. Random block projection is a popular technique to handle this kind of problem for its remarkable reduction of the computational cost from the projection. However, this powerful method has not been proposed for the situation that first-order information is prohibited and only zeroth-order information is available. In this paper, we propose to develop different classes of zeroth-order stochastic block coordinate type methods. Zeroth-order block coordinate descent (ZS-BCD) is proposed for solving unconstrained nonconvex optimization problem. For composite optimization, we establish the zeroth-order stochastic block mirror descent (ZS-BMD) and its associated two-phase method to achieve the complexity bound for the $(\eee,\Lambda)$-solution of the composite optimization problem. Furthermore, we also establish zeroth-order stochastic block coordinate conditional gradient (ZS-BCCG) method for nonconvex (composite) optimization. By implementing ZS-BCCG method, in each iteration, only (approximate) linear programming subproblem needs to be solved on a random block instead of a rather costly projection subproblem on the whole decision space, in contrast to the existing traditional stochastic approximation methods. In what follows, an approximate ZS-BCCG method and corresponding two-phase ZS-BCCG method are proposed. This is also the first time that a two-phase BCCG method has been developed to achieve the $(\eee, \Lambda)$-solution of nonconvex composite optimization problem. To the best of our knowledge, the proposed results in this paper are new in stochastic nonconvex (composite) optimization literature.
\end{abstract}

\section{Introduction}
This paper is concerned with the nonlinear stochastic optimization (NSO) problem given by
\be
f^{*}=\min_{x\in X}\bigg\{f(x)=\mbb E_{\xi}\big[F(x,\xi)\big]=\int_\Omega F(x,\xi)dP(\xi)\bigg\},\label{pro1}
\ee
and a class of stochastic composite optimization problem give by
\be
\Phi^{*}=\min_{x\in X}\big\{\Phi(x)=f(x)+\chi(x)\big\}, \label{pro2}
\ee
where $X\subseteq \mbb R^n$ is a closed convex set, $\xi$ is a random variable supported on sample space $\Omega\subseteq \mbb R^n$, $F(x,\xi)$ is a Borel measurable function on $X\times\Omega$, and almost surely for every $\xi$, $F(\cdot,\xi):X\rightarrow\mbb R$ is continuous. $f$ in \eqref{pro2} is defined as in \eqref{pro1}. $\chi$ is a convex regularization function . Optimization on nonconvex objective function $f$ and composite function $\Phi$ has played an important role in statistical machine learning and network engineering. Recent years have witnessed a resurgence of research interest on this topic following the stochastic approximation (SA) technique (see e.g. \cite{sa3,  zero first, compotp, sa5, sa2, sa1, sa4}). In this paper, we are interested in the case when the feasible set $X$ is assumed to have the block structure
\be
\nono X=X_1\times X_2\times\cdots\times X_b,\label{block structure}
\ee
in which $X\subseteq\mbb R^{n_s}$, $s=1,2,...,b$, are closed convex sets and $n_1+n_2+\cdots+n_b=n$. Also, the regularization function $\chi$ is assumed to be block separable. For solving the problem in which $X$ has the block structure \eqref{block structure}, different classes of block coordinate decent (BCD) based methods have been developed quickly (see e.g., \cite{b1, b7, rule1, b8, b3, rule2, b4, b5, b0, b2, b6}). The coordinate gradient descent method is introduced in \cite{b0}. The complexity of the BCD type methods has been widely studied in recent years (see e.g., \cite{rule1, b3,rule2, b4}). Meanwhile, many recent studies concentrate on the acceleration of the BCD type method (see e.g., \cite{b6}). Stochastic type BCD is a popular method to solve problem \eqref{pro1} in which $X$ has structure \eqref{block structure} (see e.g., \cite{rule1,b2}). This type of method depends on the access to different level of stochastic oracle information. There are three common oracles: zeroth-order oracle (function queries), first-order (gradient queries), second-order oracle (Hessian queries).

However, the current existing stochastic BCD methods are all at least first-order type. Moreover, there are increasingly more situations in which only noisy objective function information is available in statistical machine learning, and first-order information can not be accessed. Then the first-order BCD method can not be used in these situations. Meanwhile, the study of zeroth-order theory is lacking in these situations and has not been proposed in coordinate decent literature. These facts motivate us to consider the possibility of establishing  some classes of zeroth-order block coordinate type methods to solve problems \eqref{pro1}, \eqref{pro2} in these settings.

On the other hand, the classical conditional gradient (CG) method proposed by Frank and Wolfe \cite{firstCG}, has resurged in recent years (see e.g., \cite{c1, c2, c4, sliding, c3, random scheme2}). The CG method is computationally cheaper in many situations, since the CG algorithm scheme turns the optimization problem into a linear optimization subproblem rather than a costly projection to feasible set $X$. In some cases,  a general projection   might be computationally prohibited in practice.  Hence, this advantage motivates us to consider that, when the feasible set $X$ has the block structure \eqref{block structure}, is it possible to incorporate the random block decomposition with the CG scheme? i.e. In each iteration, the conditional gradient procedure is operated on only one random block. If this can be realized, we only need to solve the linear optimization subproblem on one block instead of the whole decision region $X$, further saving the computational cost. We note that this kind of block coordinate CG algorithm is still unexplored before. In this paper, we give a positive answer to the above question.

The main goal in this paper is to develop several classes of zeroth-order stochastic block coordinate type methods to solve \eqref{pro1} and \eqref{pro2}. The randomization scheme for outputting the random solution from sequence trajectory has been widely used in several recent researches (see e.g., \cite{c1, rule1, zero first, compotp, random scheme1, random scheme2}). Meanwhile, this randomization scheme is still in the practice stage of presenting convergence results for (nonconvex) stochastic optimization, it deserves to be further developed in aforementioned aspects. We will establish several theoretical results via the randomization scheme and analyze the iteration complexity of them in different aspects.  Inspired by the random selection rule in \cite{rule3,rule1,rule2}, we present a zeroth-order stochastic BCD (ZS-BCD) to solve unconstrained NSO problem \eqref{pro1}, then by incorporating the i.i.d random block selection procedure, we develop a zeroth-order stochastic block mirror descent (ZS-BMD) to solve well-known nonconvex composite optimization problem \eqref{pro2}. In what follows, a class of zeroth-order stochastic block coordinate conditional gradient method (ZS-BCCG) is proposed for solving \eqref{pro1} and \eqref{pro2}.

Two-phase optimization technique has been shown to be a powerful tool in searching for $(\epsilon, \Lambda)$-solution of NSO (see e.g., \cite{zero first, compotp, sliding}), i.e., a point $\bar{x}$ such that $Prob\{\|\nabla f(\bar{x})\|^2>\epsilon\}\leq \Lambda$ for some $\epsilon>0$ and $\Lambda\in (0,1)$. In \cite{zero first}, Ghadimi and Lan develop a two-phase stochastic first and zeroth-order method for unconstrained stochastic optimization. Then in \cite{compotp}, they further extend this technique to composite optimization scenario and develop a two-phase stochastic projected gradient method.  In \cite{sliding}, Lan and Zhou establish a two-phase conditional gradient type method via a sliding procedure to solve a convex programming problem. All these two-phase methods have the advantage of improving the direct complexity bound coming from corresponding Markov inequality on confidence level $\Lambda$. However, to the best of our knowledge, any two-phase block coordinate method has not been established for nonconvex (composite) optimization yet. Moreover, zeroth-order result in coordinate descent literature is also quite limited and needs to be explored, it makes sense to establish two-phase zeroth-order block coordinate type methods.  Inspired by above two-phase techniques, we propose to establish a two-phase ZS-BMD (2-ZS-BMD) method and two-phase ZS-BCCG (2-ZS-BCCG) method for the constrained composite nonconvex optimization problem, and provide complexity analysis for them respectively.

The contributions of this paper mainly contain the following aspects:

(\textbf{i}) Zeroth-order stochastic block decomposition is first introduced in convex and nonconvex optimization problem. In contrast to existing block coordinate type algorithms, the proposed methods in this paper remove the need for evaluation of the first-order information of the gradient that the former stochastic BCD methods used, making it potentially flexible to solve some NSO problems where the gradient is costly to evaluate. For constrained optimization problem, evaluation of a single random block component of zeroth-order oracle is implemented. Meanwhile, Zeroth-order projection or linear optimization is performed on only one random block instead of the whole $X$.  The advantage of the procedure makes the proposed methods save the iteration cost coming from the projection subproblem effectively.

(\textbf{ii}) For unconstrained stochastic optimization \eqref{pro1}, we develop a ZS-BCD method and give convergence analysis for both nonconvex and convex objective functions. We propose two new fundamental classes of stochastic block coordinate type methods (ZS-BMD with 2-ZS-BMD and ZS-BCCG with 2-ZS-BCCG) and provide complexity analysis for both methods. For ZS-BMD, we show that, to find an $\epsilon$-stationary point of composite problem \eqref{pro2}, the total number of calls to stochastic zeroth-order oracle (SZO) performed by ZS-BMD can be bounded by $\o(b^2n/\epsilon^2)$. For ZS-BCCG, we achieve an $\o(b^4n/\epsilon^4)$ complexity bound in a performance of classical Frank-Wolfe gap. We further develop an approximate ZS-BCCG algorithm and achieve an improved $\o(b^2n/\eee^2)$ complexity bound in the performance of generalized gradient. These results are new in this literature. The results give explicit dependency of complexity on both block index $b$ and dimension index $n$, which do not appear in existing  results. Meanwhile, this is the first time to establish a class of random block projection based stochastic CG method, in which the SZO is also considered.

(\textbf{iii}) To the best of our knowledge, in contrast to the existing two-phase optimization schemes established for SA algorithms. This is the first time to implement a two-phase procedure for stochastic block coordinate type method. For 2-ZS-BMD and 2-ZS-BCCG, we show that, to find an $(\epsilon,\Lambda)$-solution of problem \eqref{pro2}, the total number of calls to SZO achieve the
\be
\o\bigg\{\f{bn}{\eee}\log_2(\f{1}{\Lambda})+\f{b^2n(M^2+\sigma^2)}{\eee^2}\log_2(\f{1}{\Lambda})+\f{n(M^2+\sigma^2)}{\Lambda\eee}\log_2^2(\f{1}{\Lambda})\bigg\} \label{commain}
\ee
complexity bound (bound parameter $M$ s.t. $\|\nabla f\|\leq M$ and gradient variance parameter $\sigma$). The complexity results are new in two-phase optimization literature, they improve several recent two-phase nonconvex optimization results \cite{zero first, compotp} to block coordinate setting. In contrast to previous complexity bound, the block index factor $b$ and $b^2$ appearing in the representation accurately indicate the influence of the block projection procedure on the complexity. In addition, detailed form and analysis of the complexity on block Lipschitz estimations and other parameters are also achieved.

\textbf{Notation:} Denote the n-dimension Euclidean space by $\mathbb R^n$, let $\mathbb R^{n_s}$, $s=1, 2,..., b$ be the Euclidean spaces with norm $\|\cdot\|_i$ induced by the inner product $\nn \cdot, \cdot \mm$ such that $n_1+n_2+\cdots+n_b=n$. For a matrix $M\in\mathbb R^{n\times n}$, denote the element in $i$th row and $j$th column by $[M]_{ij}$, denote the transpose of $M$ by $M^T$. Denote the identity matrix in $\mathbb R^n$ by $I_n$  and let $U_s\in\mathbb R^{n\times n_i}$, $s=1, 2, ..., b$ be the sets of matrices such that $(U_1, U_2,... , U_b)=I_n$. For a vector $x\in\mathbb R^n$, denote its $s$th block by $x^{(s)}=U_s^Tx$, $s=1,2,...,b$. For a differentiable function $f$, denote the gradient of $f$ by $\nabla f(x)$, and the $s$th block of $\nabla f(x)$ by $\nabla_s f(x)$. Denote the class of functions which has continuous $L$-Lipschitz gradient by $C_{L}^{1,1}$. For a positive real number $c$, $\lceil c\rceil$ and $\lfloor c\rfloor$ denote the smallest integer bigger than $c$ and the biggest integer smaller than $c$.


\section{Preliminary}
The main assumptions of this paper are listed in this section. Additional assumptions are added in other section when needed.
\begin{ass}\label{ass1}
For any $x\in X$, the zeroth-order oracle outputs the estimator of $f$ such that $\mbb E[F(x,\xi)]=f(x)$, $\mbb E[\nabla F(x,\xi)]=\nabla f(x)$ and $\mbb E[\|\nabla F(x,\xi)-\nabla f(x)\|^2]\leq\sigma^2$.
\end{ass}
\begin{ass}\label{ass2}
Almost surely for any $\xi$, $F$ has the $L_f$-Lipschitz continuous gradient, i.e.,$\|\nabla F(x,\xi)-\nabla F(y,\xi)\|\leq L_f\|x-y\|$. Almost surely for any $\xi$, $F$ also satisfies the block gradient Lipschitz condition:
\begin{equation}
\|\nabla_sF(x+U_se_s,\xi)-\nabla_sF(x,\xi)\|_s\leq L_s\|e_s\|_s , s=1,2,...,b. \label{block lip F}
\end{equation}
\end{ass}
\noindent Assumption \ref{ass2} directly implies $f\in C_{L_f}^{1,1}$ and
\be
\|\nabla_sf(x+U_se_s)-\nabla_sf(x)\|_s\leq L_s\|e_s\|_s , s=1,2,...,b, \label{block lip}
\ee
where $e_s$ is the $s$th standard basis vector in $\mathbb R^{n_s}$. Now we list several results for the zeroth-order gradient estimator and smoothing function. Let $u$ be an $n$-dimensional standard Gaussian random vector and $\mu$ be the smoothing parameter. The smoothing function $f_{\mu}(x)$ of $f(x)$ is defined as the convolution of the Gaussian kernel and $f$:
\be\label{deffmu}
\nono f_{\mu}(x)=\mbb E[f(x+\mu u)]=\f{1}{(2\pi)^{\f{n}{2}}}\int f(x+\mu u)e^{-\f{1}{2}\|u\|^2}du.
\ee
Nesterov \cite{Nesterov} has shown that
\be\label{gradfmu}
\nabla f_{\mu}(x)=\f{1}{(2\pi)^{\f{n}{2}}}\int\f{f(x+\mu u)-f(x)}{\mu}ue^{-\f{1}{2}\|u\|^2}du,
\ee
and suggests to consider the stochastic gradient of $f_{\mu}(x)$ given by
\be
G_{\mu}(x,\xi,u)=\f{F(x+\mu u,\xi)-F(x,\xi)}{\mu}u,
\ee
which is an unbiased estimator of $\nabla f_{\mu}(x)$.
\noindent The main property of $f_{\mu}$ is described by the following lemma due to Nesterov \cite{Nesterov}.
\blm\label{Nes Lem}
For any $f\in C_{L_f}^{1,1}$,

(a)$f_{\mu}\in C_{L_f}^{1,1}$, $\mbb E_{\xi,u}[G_{\mu}(x, \xi, u)]=\nabla f_{\mu}(x)$.

(b)for any $x\in \mbb R^n,$
\bea
&&\ \ \ \ \ \ |f_{\mu}(x)-f(x)|\leq\f{\mu^2}{2}L_fn, \\
&&\|\nabla f_{\mu}(x)-\nabla f(x)\|\leq\f{\mu}{2}L_f(n+3)^{\f{3}{2}}.\label{fmu-f}
\eea

(c)For any $x\in\mbb R^n$,
\be
\f{1}{\mu^2}\mbb E_{u}\big[\{f(x+\mu u)-f(x)\}^2\|u\|^2\big]\leq\f{\mu^2}{2}L_f^2(n+6)^3+2(n+4)\|\nabla f(x)\|^2.
\ee
\elm
\noindent By using relation $(a+b)^2\leq2a^2+2b^2$, \eqref{fmu-f} directly implies the following estimates
\bea
&&\|\nabla_{\mu}f(x)\|^2\leq2\|\nabla f(x)\|^2+\f{\mu^2}{2}L_f^2(n+3)^2, \label{t1}\\
&&\ \ \|\nabla f(x)\|^2\leq2\|\nabla f_{\mu}(x)\|^2+\f{\mu^2}{2}L_f^2(n+3)^2. \label{t2}
\eea
These estimates will be the crucial estimates of some of the following results.

\section{Unconstrained nonconvex optimization: ZS-BCD}
In this section, we present a class of ZS-BCD for solving unconstrained nonconvex and convex optimization version of \eqref{pro1} as follow,
\begin{eqnarray}
\min_{x\in \mbb R^n} f(x),\label{pro3}
\end{eqnarray}
where we spit $\mbb R^n$ into $\mbb R^n=\mbb R^{n_1}\times\mbb R^{n_2}\cdots\times\mbb R^{n_b}$. We give a glimpse of the proposed methods in an unconstrained setting.

\subsection{Nonconvex objective function}
\begin{framed}
\noindent\textbf{ZS-BCD method:}

\noi\textbf{Input}: Initial point $x_1\in \mbb R^n$, smoothing parameter $\mu$, total iterations $T$, stepsizes $\{\aaa_k\}$, $k\geq1$, probability mass function $P_R(\cdot)$ supported on $\{1,2,...,T\}$, probabilities $p_s\in [0,1], s=1,2,...,b$, s.t. $\sum_{s=1}^bp_s=1$.

\noi\textbf{Step} 0: Generate a random variable $i_k$ according to
\begin{equation}
Prob\{i_k=s\}=p_s, s=1,2,...,b, \label{distributep}
\end{equation}
and let $R$ be a random variable with probability mass function $P_R$.

\noi\textbf{Step} k=1,2,...,R-1: Use the Guassian random vector generator to generate $u_k$ and call the SZO to compute
\begin{equation}
G_{\mu}(x_k,\xi_k,u_k)=\f{F(x_k+\mu u_k,\xi_k)-F(x_k,\xi_k)}{\mu}u_k.
\end{equation}

\noi \underline{Update $x_{k}$} by:
\begin{eqnarray}
x_{k+1}^{(s)}&=&\left\{
  \begin{array}{ll}
    x_{k}^{(i_k)}-\aaa_kG_{\mu}^{(i_k)}(x_k,\xi_k,u_k), & \hbox{$s=i_k$;} \\
    x_{k}^{(s)}, & \hbox{$s\neq i_k$.}
  \end{array}
\right.
\end{eqnarray}

\noi\textbf{Output} $x_R$.
\end{framed}
\begin{thm}\label{thm1}
Let $\{x_k\}$ be generated by ZS-BCD. Let the stepsizes $\{\aaa_k\}$ are chosen such that $\aaa_k\leq \min_{s\in\b}\{p_s\}/[2\max_{s\in\b}\{p_sL_s\}(n+4)]$ and the probability mass function $P_R$ in the ZS-BCD is defined as
\begin{equation}
 P_R(k):=Prob\{R=k\}=\f{\aaa_k[\min_{s\in\b}\{p_s\}-2(n+4)\max_{s\in\b}\{p_sL_s\}\aaa_k]}{\sum_{k=1}^T\aaa_k[\min_{s\in\b}\{p_s\}-2(n+4)\max_{s\in\b}\{p_sL_s\}\aaa_k]}, k=1,2,...,T.
\end{equation}
Then, under Assumptions \ref{ass1} and \ref{ass2}, the following gradient estimate holds:
\begin{eqnarray}\label{uncon}
\begin{aligned}
\f{1}{L_f}\mathbb E\big[\|\nabla f(x_R)\|^2\big]&\leq&\f{1}{\sum_{k=1}^T\aaa_k[\min_{s\in\b}\{p_s\}-2(n+4)\max_{s\in\b}\{p_sL_s\}\aaa_k]}\Bigg[D_f^2  \\
&&+2\f{\max_{s\in\b}\{p_sL_s\}}{L_f}(n+4)\sigma^2\sum_{k=1}^T\aaa_k^2+2\mu^2(n+4)\times \\
&&\bigg(1+L_f(n+4)^2\sum_{k=1}^T(\f{\mathop{\min}_{s\in\b}\{p_s\}}{4}\aaa_k+\max_{s\in\b}\{p_sL_s\}\aaa_k^2)\bigg)\Bigg],
\end{aligned}
\end{eqnarray}

in which the expectation is taken with respect to $R$, $\xi_{[T]}$, $u_{[T]}$, and $D_f$ is defined as $[2(f(x_1)-f^{*})/L_f]^{\f{1}{2}}$.
\end{thm}
\begin{proof}
Denote $\tau_k=(\xi_k,u_k)$, $\Delta_k=G_{\mu}(x_k,\tau_k)-\nabla f_{\mu}(x_k)$. Assumptions \ref{ass1}, \ref{ass2}, \eqref{deffmu} and Lemma \ref{Nes Lem} (a) imply that $\nabla f_{\mu}$ satisfies block Lipschitz condition with constant $L_s$, $s=1,2,...,b$. Thus the following holds for $i_k$th block,
\begin{eqnarray}
\nono f_{\mu}(x_{k+1})&\leq& f_{\mu}(x_k)+\nn\nabla_{i_k}f_{\mu}(x_k),x_{k+1}^{(i_k)}-x_k^{(i_k)}\mm+\f{L_{i_k}}{2}\|x_{k+1}^{(i_k)}-x_k^{(i_k)}\|_{i_k}^2 \\
\nono&\leq&f_{\mu}(x_k)-\aaa_k\|\nabla_{i_k}f_{\mu}(x_k)\|_{i_k}^2-\aaa_k\nn\nabla_{i_k}f_{\mu}(x_k),\Delta_k^{(i_k)}\mm+\f{\aaa_k^2}{2}L_{i_k}\|G_{\mu}^{(i_k)}(x_k,\tau_k)\|_{i_k}^2.
\end{eqnarray}
Rearranging terms, summing up above inequality from $k=1$ to $T$, using $f_{\mu}(x_{T+1})\geq f_{\mu}^{*}$, we have
\begin{equation}\label{1m}
\sum_{k=1}^T\aaa_k\|\nabla_{i_k}f_{\mu}(x_k)\|_{i_k}^2\leq f_{\mu}(x_1)-f_{\mu}^{*}-\sum_{k=1}^T\aaa_k\nn\nabla_{i_k}f_{\mu}(x_k),\Delta_k^{(i_k)}\mm+\sum_{k=1}^T\f{\aaa_k^2}{2}L_{i_k}\|G_{\mu}^{(i_k)}(x_k,\tau_k)\|_{i_k}^2.
\end{equation}
Note that
\begin{equation}
\nono\mathbb E_{i_{[k-1]},\tau_{[k-1]}}\big[\nn\nabla_{i_k}f_{\mu}(x_k),\Delta_k^{(i_k)}\mm\big]=\sum_{s=1}^bp_s\mathbb E\big[\nn\nabla_sf_{\mu}(x_k),\Delta_k^{(s)}\big\mm]=0,
\end{equation}
which follows from measurability of $\nabla_sf_{\mu}(x_k)$ with respect to the history $\tau_{[k-1]}$ and $\mathbb E_{\tau_{[k-1]}}[\Delta_k]=0$, and
\begin{eqnarray}
\nono&&\mathbb E_{i_{[k-1]}\tau_{[k-1]}}\big[L_{i_k}\|G_{\mu}^{(i_k)}(x_k,\tau_k)\|_{i_k}^2\big]=\sum_{s=1}^bp_sL_s\mathbb E_{\tau_{[k-1]}}[\|G_{\mu}^{(s)}(x_k,\tau_k)\|_{s}^2]  \\
\nono&&\leq\max_{s\in\b}\{p_sL_s\}\mathbb E_{\tau_{[k-1]}}[\|G_{\mu}(x_k,\tau_k)\|^2] \\
\nono&&\leq\max_{s\in\b}\{p_sL_s\}\bigg[2(n+4)\mathbb E_{\tau_{[k-1]}}[\|\nabla f(x_k)\|^2]+2(n+4)\sigma^2+\f{\mu^2}{2}L_f^2(n+6)^3\bigg],
\end{eqnarray}
in which the second inequality follows from Lemma \ref{Nes Lem} (c) and Assumption \ref{ass1}. Also note that
\begin{eqnarray}
\nono&&\mathbb E_{i_{[k-1]}\tau_{[k-1]}}\big[\|\nabla_{i_k}f_{\mu}(x_k)\|_{i_k}^2\big]=\sum_{s=1}^bp_s\mathbb E_{\tau_{[k-1]}}[\|\nabla_s f_{\mu}(x_k)\|_{s}^2] \\
\nono&&\geq\min_{s\in\b}p_s\mathbb E_{\tau_{[k-1]}}\big[\|\nabla f_{\mu}(x_k)\|^2\big] \geq\min_{s\in\b}\f{p_s}{2}\bigg[\mathbb E_{\tau_{[k-1]}}[\|\nabla f(x_k)\|^2]-\f{\mu^2}{2}L_f^2(n+3)^3\bigg].
\end{eqnarray}
Take expectation on both sides of \eqref{1m} with respect to $i_{[T]}$, $\tau_{[T]}$,  it follows that
\begin{eqnarray}
\nono&&\sum_{k=1}^T\aaa_k\f{\min_{s\in\b }p_s}{2}\bigg[\mathbb E_{i_{[T]},\tau_{[T]}}\|\nabla f(x_k)\|^2-\f{\mu^2}{2}L_f^2(n+3)^3\bigg]\leq f_{\mu}(x_1)-f_{\mu}^{*}+\sum_{k=1}^T\f{\aaa_k^2}{2}\max_{s\in\b}\{p_sL_s\} \\
\nono&&\times\bigg[2(n+4)\mathbb E_{i_{[T]},\tau_{[T]}}[\|\nabla f(x_k)\|^2]+2(n+4)\sigma^2+\f{\mu^2}{2}L_f^2(n+6)^3\bigg].
\end{eqnarray}
Rearranging terms, we obtain that
\begin{eqnarray}
\nono&&\sum_{k=1}^T\aaa_k[\min_{s\in\b}p_s-2(n+4)\max_{s\in\b}\{p_sL_s\}\aaa_k]\mbb E_{i_{[T]},\tau_{[T]}}[\|\nabla f(x_k)\|^2]  \\
\nono&&\leq2[f(x_1)-f^{*}]+2\mu^2L_fn \\
\nono&&\ \ \ +\max_{s\in\b}\{p_sL_s\}\bigg[2(n+4)\sigma^2+\f{\mu^2}{2}L_f(n+6)^3\bigg]\sum_{k=1}^T\aaa_k^2\\
\nono&&\ \ \ +\f{1}{2}\min_{s\in\b}\{p_s\}\cdot\mu^2L_f^2(n+3)^3\sum_{k=1}^T\aaa_k.  \\
\nono&&\leq2[f(x_1)-f^{*}]+2\max_{s\in\b}\{p_sL_s\}(n+4)\sigma^2\sum_{k=1}^T\aaa_k^2\\
\nono&&\ \ \ +2\mu^2L_f(n+4)\Bigg[1+L_f(n+4)^2\sum_{k=1}^T\bigg(\f{\min_{s\in\b}\{p_s\}}{4}\aaa_k+\max_{s\in\b}\{p_sL_s\}\aaa_k^2\bigg)\Bigg].
\end{eqnarray}
Divide both sides by $L_f$ and note the definition of  $D_f$, the desired result is obtained.
\end{proof}
In the rest of the paper, we denote $\hat{L}=\max_{s\in\b}L_s$.
\begin{cor}\label{corunc}
Under assumptions of Theorem \ref{thm1}. Suppose the random variables $\{i_k\}$ are uniformly distributed ($p_1=p_2=\cdots=1/b$). Let the stepsizes $\{\aaa_k\}$ be selected as
\be
\nono \aaa_k=\f{1}{\sqrt{n+4}}\min\bigg\{\f{\w{D}}{\sigma\sqrt{T}},\f{1}{4\hat{L}(n+4)}\bigg\}, \ k=1,2,...,T,
\ee
$\w{D}$ is some positive parameter. The smoothing parameter $\mu$ is selected such that
\be
\nono\mu\leq\f{D_f}{n+4}\sqrt{\f{1}{T}}.
\ee
Then we have $\mbb E[\|\nabla f(x_R)\|^2]\leq b\cdot L_f\textbf{B}_T$, where
\be
\textbf{B}_T=\f{2\sigma\sqrt{n+4}}{\sqrt{T}}\big(\f{2\hat{L}}{L_f}\w{D}+\f{3D_f^2}{\w{D}}\big)+\f{D_f^2(24\hat{L}+2L_f)(n+4)}{T}.\label{rate1}
\ee
\end{cor}
\begin{proof}
Note that when $i_k$ are uniformly distributed, and $\{\aaa_k\}$, $\mu$ are selected as above, it follows that,
\be
\sum_{k=1}^T\aaa_k[\min_{s\in\b}\{p_s\}-2(n+4)\max_{s\in\b}\{p_sL_s\}\aaa_k]\geq \f{T\aaa_1}{2b}.
\ee
Also note that $b\geq1$, after substituting above $\aaa_k$ and $\mu$ into \eqref{uncon}, it follows that
\bea
\nono&&\f{1}{L_f}\mbb E[\|\nabla f(x_R)\|^2]\\
\nono&&\leq\f{2bD_f^2+4b\mu^2(n+4)}{T\aaa_1}+\mu^2L_f(n+4)^3+4\hat{L}(n+4)\mu^2L_f\aaa_1+\f{4\hat{L}(n+4)\sigma^2}{L_f}\aaa_1  \\
\nono&&\leq\bigg(\f{2bD_f^2}{T}+\f{4bD_f^2}{(n+4)T^2}\bigg)\max\bigg\{\f{\sigma\sqrt{(n+4)T}}{\w{D}},4\hat{L}(n+4)\bigg\}+\f{bD_f^2L_f(n+5)}{T}+\f{4\hat{L}\sqrt{n+4}\sigma\w{D}}{L_f\sqrt{T}}    \\
\nono&&\leq\bigg(\f{2bD_f^2}{T}+\f{4bD_f^2}{(n+4)T^2}\bigg)\bigg[\f{\sigma\sqrt{(n+4)T}}{\w{D}}+4\hat{L}(n+4)\bigg]+\f{bD_f^2L_f(n+5)}{T}+\f{4\hat{L}\sqrt{n+4}\sigma\w{D}}{L_f\sqrt{T}}\\
\nono&&=\f{2bD_f^2\sigma\sqrt{n+4}}{\w{D}\sqrt{T}}+\f{4bD_f^2\sigma\sqrt{n+4}}{\w{D}(n+4)T\cdot\sqrt{T}}+\f{8bD_f^2\hat{L}(n+4)}{T}+\f{16bD_f^2\hat{L}}{T^2}+\f{D_f^2L_f(n+5)}{T}+\f{4\hat{L}\sqrt{n+4}\sigma\w{D}}{L_f\sqrt{T}}\\
\nono&&=\f{2b\sigma\sqrt{n+4}}{\sqrt{T}}\bigg(\f{D_f^2}{\w{D}}+\f{2D_f^2}{\w{D}(n+4)T}+\f{2\hat{L}\w{D}}{L_f}\bigg)+\f{bD_f^2(n+4)}{T}\bigg(8\hat{L}+\f{16\hat{L}}{T}+2L_f\bigg),
\eea
which gives \eqref{rate1} after noting that $(n+4)T\geq1$ and $T\geq1$.
\end{proof}
When the parameter $\w{D}$ in \eqref{rate1} is chosen as an optimal value $(3L_f/2\hat{L})^{\f{1}{2}}D_f$, an improved bound
\be
O\big(\f{\sigma D_f\sqrt{n\hat{L}/L_f}}{\sqrt{T}}+\f{nD_f^2(\hat{L}+L_f)}{T}\big) \label{convbdd1}
\ee
can be obtained for $\textbf{B}_T$. The $T$-rate for ZS-BCD in \eqref{convbdd1} matches the optimal rate for nonconvex smooth NSO problem. Corollary \ref{corunc} and Markov inequality imply that
\be
Prob\big\{\|\nabla f(x_R)\|^2\geq\lambda bL_f\textbf{B}_T\big\}\leq\f{1}{\lambda},\ \forall\lambda\geq0. \label{markov}
\ee
Also, \eqref{markov} can provide the complexity for computing an $(\eee,\Lambda)$-solution of problem \eqref{pro1} in a single run of ZS-BCD, revealing the large-deviation property for ZS-BCD. For any $\eee>0$ and $\Lambda\in(0,1)$, by setting $\lambda=1/\Lambda$ and
\be
\nono T=\bigg\lceil\max\bigg\{\f{2bL_fD_f^2(24\hat{L}+2L_f)(n+4)}{\Lambda\eee},\f{16b^2(n+4)\sigma^2L_f^2}{\Lambda^2\eee^2}\bigg(\f{2\hat{L}}{L_f}\w{D}+\f{3D_f^2}{\w{D}}\bigg)^2\bigg\}\bigg\rceil,
\ee
the complexity for finding an $(\eee,\Lambda)$-solution in ZS-BCD method, after disregarding several constant factors, can be bounded by
\be
\o\bigg\{\f{b^2n\sigma^2L_f^2}{\Lambda^2\eee^2}\bigg(\f{\hat{L}}{L_f}\w{D}+\f{D_f^2}{\w{D}}\bigg)^2+\f{bnL_f(\hat{L}+L_f)D_f^2}{\Lambda\eee}\bigg\}. \label{comb1}
\ee
The above representation highlights the dependency of the complexity bound of ZS-BCD on the block index $b$ and block Lipschitz characteristic parameter $\hat{L}$. According to the selection rule of $\{\aaa_k\}$ and $P_R$ in ZS-BCD, the influence of the estimation of $L_s$, $s=1,2,...,b$ on $\{\aaa_k\}$ and $P_R$, thus on complexity bound, seems inevitable. Meanwhile, due to the random block projection structure of the proposed algorithm, first and second order indices $b$, $b^2$ appearing in the representation \eqref{comb1} also affect the complexity. Hence, it would be a possible future research topic to reduce the effect of $b$, $\hat{L}$ on complexity of block coordinate type method. However, to reduce the effect of index $b$, we may need to improve the algorithm structure or develop other types of block coordinate type methods. To reduce effect of $L_s$, $s=1,2,..,b$, there seem to be more technical difficulties to overcome the classical handling on $C_{L_f}^{1,1}$ condition of $f$, which results in the dependency of $\{\aaa_k\}$ and $P_R$ on $\hat{L}$.

\subsection{Convex objective function}
We end this session with a convex result. Define the weighted summation $\mathcal{N}_k^2$ as follow:
\begin{equation}
\mathcal{N}_k^2=\sum_{s=1}^b\f{1}{p_s}\|x_k^{(s)}-x_{*}^{(s)}\|_s^2, \ k\geq1.   \label{pai}
\end{equation}

\begin{thm}\label{convexc}
Let $\mathcal{N}_k^2$ be defined as \eqref{pai}, denote $\mathcal{D}_{p,X}^2=\mathcal{N}_1^2$. Suppose the objective in problem \eqref{pro3} is convex with an optimal point $x_{*}$. If the stepsizes $\{\aaa_k\}$ and the probability mass function $P_R(\cdot)$ are chosen such that $\aaa_k\leq1/[2(n+5)L_f]$ and
\begin{equation}
\nono P_R(k):=Prob\{R=k\}=\f{\aaa_k-4(n+5)L_f\aaa_k^2}{\sum_{k=1}^T[\aaa_k-4(n+5)L_f\aaa_k^2]}, \ k=1,2,...,T.
\end{equation}
then for $T\geq 1$, we have
\begin{eqnarray}
\nono\mathbb E[f(x_R)-f^{*}]&\leq&\f{1}{2\sum_{k=1}^T[\aaa_k-4(n+5)L_f\aaa_k^2]}\Bigg[\mathcal{D}_{p,X}^2+2\mu^2L_f(n+5)\sum_{k=1}^T\aaa_k\\
\nono&&+8(n+5)\big[\mu^2L_f^2(n+5)^3+\mu^2L_f^2(n+5)+\sigma^2\big]\sum_{k=1}^T\aaa_k^2\Bigg]
\end{eqnarray}
in which the expectation is taken with respect to $R$, $\xi_{[T]}$, $u_{[T]}$.
\end{thm}
\begin{proof}
The ZS-BCD algorithm implies that, when $s=i_k$,
\begin{eqnarray}
\nono&&\|x_{k+1}^{(i_k)}-x_{*}^{(i_k)}\|_{i_k}^2=\|x_{k}^{(i_k)}-\aaa_kG_{\mu}^{(i_k)}(x_k,\tau_k)-x_{*}^{(i_k)}\|_{i_k}^2\\
\nono&&=\|x_{k}^{(i_k)}-x_{*}^{(i_k)}\|_{i_k}^2-2\aaa_k\nn G_{\mu}^{(i_k)}(x_k,\tau_k),x_{k}^{(i_k)}-x_{*}^{(i_k)}\mm+\|G_{\mu}^{(i_k)}(x_k,\tau_k)\|_{i_k}^2.
\end{eqnarray}
Then it follows
\begin{eqnarray}
\nono&&\mathcal{N}_{k+1}^2=\sum_{s=1}^b\f{1}{p_s}\|x_{k+1}^{(s)}-x_{*}^{(s)}\|_{s}^2=\sum_{s\neq i_k}\f{1}{p_s}\|x_{k}^{(s)}-x_{*}^{(s)}\|_{s}^2+\f{1}{p_{i_k}}\|x_{k+1}^{(i_k)}-x_{*}^{(i_k)}\|_{i_k}^2\\
\nono&&\ \ \ \ \ \ \  = \sum_{s\neq i_k}\f{1}{p_s}\|x_{k}^{(s)}-x_{*}^{(s)}\|_{s}^2+\f{1}{p_{i_k}}\bigg[\|x_{k}^{(i_k)}-x_{*}^{(i_k)}\|_{i_k}^2\\
\nono&&\ \ \ \ \ \ \ \ \ \ \ \ -2\aaa_k\nn G_{\mu}^{(i_k)}(x_k,\tau_k),x_{k}^{(i_k)}-x_{*}^{(i_k)}\mm+\|G_{\mu}^{(i_k)}(x_k,\tau_k)\|_{i_k}^2\bigg]\\
\nono&&\ \ \ \ \ \ \  =\mathcal{N}_k^2-2\aaa_k\nn\f{1}{p_{i_k}}U_{i_k}G_{\mu}^{(i_k)}(x_k,\tau_k)-\nabla f_{\mu}(x_k),x_k-x_{*}\mm\\
\nono&&\ \ \ \ \ \ \ \ \ \ \ \ -2\aaa_k\nn\nabla f_{\mu}(x_k),x_k-x_{*}\mm+\aaa_k^2\f{1}{p_{i_k}}\|G_{\mu}^{(i_k)}(x_k,\tau_k)\|_{i_k}^2  \label{pai2}
\end{eqnarray}
By taking summation from $k=1$ to $k=T$ on both sides of the above equality, we obtain
\begin{eqnarray}\label{2m}
\begin{split}
&&\mathcal{N}_{T+1}^2=\mathcal{N}_1^2-2\sum_{k=1}^T\aaa_k\nn\f{1}{p_{i_k}}U_{i_k}G_{\mu}^{(i_k)}(x_k,\tau_k)-\nabla f_{\mu}(x_k),x_k-x_{*}\mm \\
&&\ \ \ \ \ \ \ \ \ \ \ -2\sum_{k=1}^T\aaa_k\nn\nabla f_{\mu}(x_k),x_k-x_{*}\mm+\sum_{k=1}^T\aaa_k^2\f{1}{p_{i_k}}\|G_{\mu}^{(i_k)}(x_k,\tau_k)\|_{i_k}^2.
\end{split}
\end{eqnarray}
Observe that
\begin{eqnarray}
\nono &&\mathbb E_{i_{[k-1]},\tau_{[k-1]}}\big[\nn p_{i_k}^{-1}U_{i_k}G_{\mu}^{(i_k)}(x_k,\tau_k)-\nabla f_{\mu}(x_k),x_k-x_{*}\mm\big] \\
\nono&&=\sum_{s=1}^b\mathbb E_{\tau_{[k-1]}}\big[\nn U_{s}G_{\mu}^{(s)}(x_k,\tau_k)-\nabla f_{\mu}(x_k),x_k-x_{*}\mm\big]=0,
\end{eqnarray}
which follows from Lemma \ref{Nes Lem} (a), and
\begin{eqnarray}
\nono &&\mathbb E_{i_{[k-1]},\tau_{[k-1]}}\big[p_{i_k}^{-1}\|G_{\mu}^{(i_k)}(x_k,\tau_k)\|_{i_k}^2\big]=\sum_{s=1}^bp_{s}\cdot p_s^{-1}\mathbb E_{\tau_{[k-1]}}\big[\|G_{\mu}^{(s)}(x_k,\tau_k)\|_{s}^2\big] \\
\nono&&=\mathbb E_{\tau_{[k-1]}}\big[\|G_{\mu}(x_k,\tau_k)\|^2\big]\leq2(n+4)\mathbb E_{\tau_{[k-1]}}[\|\nabla f(x_k)\|^2]+2(n+4)\sigma^2+\f{\mu^2}{2}L_f^2(n+6)^3.
\end{eqnarray}
Take total expectation with respect to $i_{[T]}$ and $\tau_{[T]}$ on both sides of \eqref{2m}, use the notation $\mathcal{D}_{p,X}^2=\mathcal{N}_1^2$, it follows that
\begin{eqnarray}
\nono&&\mathbb E_{i_{[T]},\tau_{[T]}}[\mathcal{N}_{T+1}^2]\leq \mathcal{D}_{p,X}^2-2\sum_{k=1}^T\aaa_k\mathbb E_{i_{[T]},\tau_{[T]}}[\nn\nabla f_{\mu}(x_k),x_k-x_{*}\mm]+2(n+4)\sum_{k=1}^T\aaa_k^2\mathbb E_{i_{[T]},\tau_{[T]}}[\|\nabla f(x_k)\|^2]\\
\nono&&\ \ \ \ \ \ \ \ \ \ \ \ \ \ \ \ \ \  \ \ \ \ \  \ +\bigg[2(n+4)\sigma^2+\f{\mu^2}{2}L_f(n+6)^3\bigg]\sum_{k=1}^T\aaa_k^2. \label{dai1}
\end{eqnarray}
Noting that
\bea
\nono&&\|\nabla f(x_k)\|^2\leq2\|\nabla f_{\mu}(x_k)\|^2+\f{\mu^2}{2}L_f^2(n+3)^3\\
\nono&&\ \ \ \ \ \ \ \ \ \ \ \ \ \ \ \leq4\|\nabla f_{\mu}(x_k)-\nabla f_{\mu}(x_{*})\|^2+4\|\nabla f_{\mu}(x_{*})\|^2+\f{\mu^2}{2}L_f^2(n+3)^3 \\
\nono&&\ \ \ \ \ \ \ \ \ \ \ \ \ \ \ \leq4L_f\nn\nabla f_{\mu}(x_k),x_k-x_{*}\mm+\f{5\mu^2}{2}L_f^2(n+3)^3, 
\eea
in which the first inequality follows from \eqref{t2}, and the third inequality follows from the convexity and $L_f$-Lipschitz property of $f_{\mu}$ and Lemma \ref{Nes Lem} with $x=x_*$. Combining the above two inequalities, we obtain
\bea
\nono &&\mbb E_{i_{[T]},\tau_{[T]}}[\mathcal{N}_{T+1}^2]\leq\mathcal{D}_{p,X}^2-2\sum_{k=1}^T\big(\aaa_k-4(n+4)L_f\aaa_k^2\big)\mbb E_{i_{[T]},\tau_{[T]}}[\nn\nabla f_{\mu}(x_k),x_k-x_{*}\mm] \\
\nono &&\ \ \ \ \ \ \ \ \ \ \ \ \ +\bigg[5\mu^2L_f^2(n+4)(n+3)^3+2(n+4)\sigma^2+\f{\mu^2}{2}L_f^2(n+6)^3\bigg]\sum_{k=1}^T\aaa_k^2  \\
\nono&&\ \ \ \ \ \ \ \ \ \ \   \leq \mathcal{D}_{p,X}^2-2\sum_{k=1}^T\big(\aaa_k-4(n+4)L_f\aaa_k^2\big)\mbb E_{i_{[T]},\tau_{[T]}}[f_{\mu}(x_k)-f_{\mu}(x_{*})]  \\
\nono&&\ \ \ \ \ \ \ \ \ \ \ \ \ +8(n+4)\big[\mu^2L_f^2(n+4)^2(n+5)+\sigma^2)\big]\sum_{k=1}^T\aaa_k^2 \\
\nono&&\ \ \ \ \ \ \ \ \ \ \  \leq\mathcal{D}_{p,X}^2-2\sum_{k=1}^T\big(\aaa_k-4(n+4)L_f\aaa_k^2\big)\mbb E_{i_{[T]},\tau_{[T]}}[f(x_k)-f^{*}-\mu^2L_f n] \\
\nono&&\ \ \ \ \ \ \ \ \ \ \ \ \ +8(n+5)\big[\mu^2L_f^2(n+5)^3+\sigma^2\big]\sum_{k=1}^T\aaa_k^2.
\eea
Rearranging terms and simplifying the coefficients, we obtain
\bea
\nono &&2\sum_{k=1}^T[\aaa_k-4(n+5)L_f\aaa_k^2]\mbb E_{i_{[T]},\tau_{[T]}}[f(x_k)-f^{*}]\leq\mathcal{D}_{p,X}^2+2\mu^2L_f(n+5)\sum_{k=1}^T\aaa_k \\
\nono&&\ \ \ \ \ \ \ \ \ \ \ \ +8(n+5)\big[\mu^2L_f^2(n+5)^3+\mu^2L_f^2(n+5)+\sigma^2\big]\sum_{k=1}^T\aaa_k^2.
\eea
Then the desired result is obtained by noting the definition of $P_R(\cdot)$.
\end{proof}
In Theorem \ref{convexc}, if the stepsizes $\{\aaa_k\}$ are taken as $\aaa_k=\f{1}{\sqrt{n+5}}\min\big\{\f{\w{D}}{\sigma\sqrt{T}},\f{1}{8L_f(n+5)}\big\}$, $k=1,2,...,T$,
the smoothing parameter $\mu$ are selected such that $\mu\leq\mathcal{D}_{p,X}/\sqrt{n+5}$. After selecting an optimal value of parameter $\w{D}$, an
\be
O\big(\f{\sigma\mathcal{D}_{p,X}\sqrt{n}}{\sqrt{T}}+\f{n\mathcal{D}_{p,X}^2L_f}{T}\big) \label{convbdd2}
\ee
bound can be obtained for $\mbb E[f(x_R)-f^{*}]$. The proving details are similar with the procedures in Corollary \ref{corunc}, we omit them for saving space. In contrast to the bound \eqref{convbdd1} for nonconvex case, the disappearance of $\hat{L}$ shows that the convex case removes the needs of making $L_s$-estimation, $s=1,2,...,b$ to achieve the rate bound. In addition, weighted summation parameter $\mathcal{D}_{p,X}$ highlights the importance of block coordinate projection in convex case. It deserves to be explored that, if the rate of \eqref{convbdd2} can be accelerated to an optimal $T$-rate of $O\big(\f{\sigma\mathcal{D}_{p,X}\sqrt{n}}{\sqrt{T}}+\f{n\mathcal{D}_{p,X}^2L_f}{T^2}\big)$ by incorporating with accelerated gradient technique \cite{random scheme1},\cite{acc1}?
\section{Constrained composite optimization: ZS-BMD}
\subsection{ZS-BMD method}
In this section, a general problem setting is considered, we proposed to develop ZS-BMD method to solve the constrained nonconvex composite optimization problem \eqref{pro2}. The constrained decision domain $X$ has the block structure $X=X_1\times X_2\times\cdots\times X_b$, in which $X_s\in \mathbb R^{n_s}$, $s=1, 2,..., b$, are closed convex sets and $n_1+n_2+\cdots+n_b=n$. In the rest of the paper, function $\chi$ is assumed to be block separable, which means that $\chi(\cdot)$ can be written into the form $\chi(x)=\sum_{s=1}^b\chi_s(x^{(s)})$ for any $x\in X$, in which $\chi_s:\mbb R^{n_s}\rightarrow\mbb R$, $s=1,2,...,b$ are closed and convex. We introduce several standard notations for later use. For $x,g\in\mbb R^n$, $\aaa>0$, define $P(x,g,\aaa)=(P_1(x^{(1)},g^{(1)},\aaa),P_2,(x^{(2)},g^{(2)},\aaa),...,P_b(x^{(b)},g^{(b)},\aaa)$ in which
\bea
P_s(x^{(s)},g^{(s)},\aaa)=\arg\min_{y\in X_s}\big\{\nn g^{(s)},y\mm+\f{1}{\aaa}D_{\phi_s}(y,x^{(s)})+\chi_s(y)\big\}, \ s=1,2,...,b. \label{P}
\eea
denote $\p(x,g,\aaa)=\big(\p_1(x,g,\aaa),\p_2(x,g,\aaa),...,\p_s(x,g,\aaa)\big)$, in which the generalized block gradient mapping is defined as
\be
\p_s(x,g,\aaa)=\f{1}{\aaa}\big[U_s^Tx-P_s(U_s^Tx,U_s^Tg,\aaa)\big], \ s=1,2,...,b. \label{PP}
\ee
$D_{\phi_s}(x,y)=\phi_s(x)-\phi_s(y)-\nn\nabla\phi_s(y),x-y\mm$ is the Bregman divergence with distance generating function $\phi_s$. Without loss of generality, in what follows we assume $\phi_s$, $s=1,2,...,b$ to be $1$-strongly convex functions. We mention that, when $X=\mbb R^n$ with $X_s=\mbb R^{n_s}$, $s=1,2,...,b$, $g=\nabla f(x)$, $\chi_s=0$, and $\phi_s(x)=\f{1}{2}\|x\|_s^2$, the generalized block gradient degenerates to classical block gradient $\nabla_s f(x)$ in last section, and the results in this section can be considered as extensions of the last section in many aspects. To get the main result, we need several basic lemmas of $P(x,g,\aaa)$ and $\p(x,g,\aaa)$, which are as results of the optimality condition of block projection operator $P_s$. In the rest of the paper, we use $\b$ to denote the index set $\{1,2,...,b\}$. For an arbitrary given $\eee>0$, call $\bar{x}$ the $\eee$-stationary point of the composite problem \eqref{pro2} if $\mbb E[\|\p(\bar{x},\nabla f(\bar{x}), \aaa)\|^2]\leq\eee$ for some $\aaa>0$. Also denote $\Phi_{\mu}(x)=f_{\mu}(x)+\chi(x)$. We make an additional assumption as follow on gradient of function $f$ in the rest of the paper.
\begin{ass}\label{ass3}
There exists a constant such that $\|\nabla f(x)\|\leq M$ for all $x\in X$.
\end{ass}

\blm\label{lemma1}
Define $P(x,g,\aaa)$ and $\p(x,g,\aaa)$ as \eqref{P} and $\eqref{PP}$, denote $x^{+}=P(x,g,\aaa)$, then
\be
\nono\nn g^{(s)},\p_s(x,U_sg^{(s)},\aaa)\mm\geq\|\p_s(x,U_sg^{(s)},\aaa)\|_s^2+\f{1}{\aaa}[\chi_s(U_s^Tx^{+})-\chi_s(U_s^Tx)], s=1,2,...,b.
\ee
\elm
\begin{proof}
See Appendix.
\end{proof}
\blm\label{lemmap}
Let $P(x,g,\aaa)$ with its components be defined as in \eqref{P}, then for any $g_1,g_2\in \mbb R^n$, we have
\be
\nono\|P_s(U_s^Tx,U_s^Tg_1,\aaa)-P_s(U_s^Tx,U_s^Tg_2,\aaa)\|_s\leq\aaa\|U_s^Tg_1-U_s^Tg_2\|_s, s=1,2,...,b. \label{lemp}
\ee
\elm
\begin{proof}
See Appendix.
\end{proof}
\blm\label{lemma2}
Let $P(x,g,\aaa)$ and $\p(x,g,\aaa)$ be defined as \eqref{P} and $\eqref{PP}$, for any $g_1,g_2\in \mbb R^n$ and any $s\in\b$, the following block non-expansion property for $\p$ holds,
\be
\nono \|\p_s(x,g_1,\aaa)-\p_s(x,g_2,\aaa)\|_s\leq\|U_s^Tg_1-U_s^Tg_2\|_s, s=1,2...,b.
\ee
\elm
\begin{proof}
See Appendix.
\end{proof}
In the rest of this paper, we make use of a popular averaging technique for zeroth-order stochastic gradient estimator (see e.g.\cite{r1}). In each step, an averaged estimator of the batch samples is used instead of an individual estimator. This procedure results in the variance reduction and the increasing precision, that is convenient for stochastic optimization circumstances.
\begin{framed}
\noindent\textbf{ZS-BMD method}:

\noi\textbf{Input}: Initial point $x_1\in X_1\times X_2\times\cdots\times X_b$, total iteration $T$, the stepsizes $\aaa_k>0$, $k\geq1$, batch sizes $T_k$ with $T_k>0$, $k\geq1$, probability mass function $P_R(\cdot)$ supported on $\{1,2,...,T\}$, probabilities $p_s\in [0,1], s=1,2,...,b$, s.t. $\sum_{s=1}^bp_s=1$.

\noi\textbf{Step} 0: Generate a random variable $i_k$ according to
\begin{equation}
\nono Prob\{i_k=s\}=p_s, s=1,2,...,b.
\end{equation}
and let $R$ be a random variable with probability mass function $P_R$.

\noi\textbf{Step} k=1,2,...,R-1: Generate $u_k=[u_{k,1},...,u_{k,T_k}]$, in which $u_{k,j}\sim N(0,I_n)$ and call the stochastic oracle to compute the $i_k$th block average stochastic gradient $G_{\mu,k}^{(i_k)}$ by
\begin{equation}
G_{\mu,k}^{(i_k)}=\f{1}{T_k}\sum_{t=1}^{T_k}U_{i_k}^TG_{\mu}(x_k,\xi_{k,t},u_{k,t}). \label{minig}
\end{equation}

\noi \underline{Update $x_{k}$} by:
\begin{eqnarray}
x_{k+1}^{(s)}&=&\left\{
  \begin{array}{ll}
    P_{i_k}(x_k^{(i_k)},G_{\mu,k}^{(i_k)},\aaa_k), & \hbox{$s=i_k$;} \\
    x_{k}^{(s)}, & \hbox{$s\neq i_k$.}     \label{BMD}
  \end{array}
\right.
\end{eqnarray}

\noi\textbf{Output} $x_R$.
\end{framed}

\begin{thm}\label{general main}
Under Assumptions \ref{ass1}, \ref{ass2}, \ref{ass3}. Suppose the stepsizes $\{\aaa_k\}$, $k\geq1$ in ZS-BMD satisfy $\aaa_k\leq2/L_s$, $s\in\b$. The probability mass function $P_R$ is chosen such that
\be
P_R(k):=Prob\big\{R=k\big\}=\f{\aaa_k\min_{s\in\b}\{ p_s(1-\f{L_s}{2}\aaa_k)\}}{\sum_{k=1}^T\aaa_k\min_{s\in\b} \{p_s(1-\f{L_s}{2}\aaa_k)\}}, \ k=1,2,...,T. \label{generalp}
\ee
Then we have
\be
\mbb E\big[\|\p(x_R,\nabla f(x_R),\aaa_R)\|^2\big]\leq\f{4[\Phi_{\mu}(x_1)-\Phi_{\mu}^{*}]+8(\max_{s\in\b}p_s)\w{\sigma}^2\sum_{k=1}^T\f{\aaa_k}{T_k}}{\sum_{k=1}^T\aaa_k\min_{s\in\b} p_s(1-\f{L_s}{2}\aaa_k)}+\f{\mu^2}{2}L_f^2(n+3)^3,  \label{generalin}
\ee
in which $\w{\sigma}^2=4(n+4)[2M^2+\sigma^2+\mu^2L_f^2(n+4)^2]$, the total expectation is taken with respect to $R$, $i_{[T]}$, $\xi_{[T]}$ and $u_{[T]}$.
\end{thm}

\begin{proof}
Denote $\w{\Delta}_k=G_{\mu,k}^{(i_k)}-\nabla_{i_k} f_{\mu}(x_k)$, then block $L_i$-Lipschitz property of $f_{\mu}$ implies
\bea
\nono&&f_{\mu}(x_{k+1})\leq f_{\mu}(x_k)+\nn\nabla f_{\mu}(x_k),x_{k+1}-x_k\mm+\f{L_{i_k}}{2}\|x_{k+1}-x_k\|^2  \\
\nono&&=f_{\mu}(x_k)-\aaa_k\nn \nabla_{i_k}f_{\mu}(x_k),\p_{i_k}(x_k,U_{i_k}G_{\mu,k}^{(i_k)},\aaa_k)\mm+\f{L_{i_k}}{2}\aaa_k^2\|\p_{i_k}(x_k,U_{i_k}G_{\mu,k}^{(i_k)},\aaa_k)\|_{i_k}^2 \\
\nono&&=f_{\mu}(x_k)-\aaa_k\nn G_{\mu,k}^{(i_k)},\p_{i_k}(x_k,U_{i_k}G_{\mu,k}^{(i_k)},\aaa_k)\mm +\f{L_{i_k}}{2}\aaa_k^2\|\p_{i_k}(x_k,U_{i_k}G_{\mu,k}^{(i_k)},\aaa_k)\|_{i_k}^2\\
\nono&&\ \ \ \ \  +\aaa_k\nn\w{\Delta}_k,\p(x_k,U_{i_k}G_{\mu,k}^{(i_k)},\aaa_k)\mm
\eea
Using Lemma \ref{lemma1} with $s=i_k$, $x^{+}=x_{k+1}$, $g^{(s)}=G_{\mu,k}^{(s)}$ to above inequality, we have
\bea
\nono&&f_{\mu}(x_{k+1})\leq f_{\mu}(x_k)-\big[\aaa_k\|\p_{i_k}(x_k,U_{i_k}G_{\mu,k}^{(i_k)},\aaa_k)\|_{i_k}+\chi_{i_k}(U_{i_k}^Tx_{k+1})-\chi_{i_k}(U_{i_k}^Tx_k)\big]\\
\nono&&\ \ \ \ \ \ \ \ \ \ \ \ \ \ \ \ \ +\f{L_{i_k}}{2}\aaa_k^2\|\p_{i_k}(x_k,U_{i_k}G_{\mu,k}^{(i_k)},\aaa_k)\|_{i_k}^2+\aaa_k\nn\w{\Delta}_k,\p_{i_k}(x_k,U_{i_k}G_{\mu,k}^{(i_k)},\aaa_k)\mm.
\eea
Note that $\chi_{i_k}(U_{i_k}^Tx_{k+1})-\chi_{i_k}(U_{i_k}^Tx_k)=\chi(x_{k+1})-\chi(x_k)$ and denote $\Phi_{\mu}=f_{\mu}(x)+\chi(x)$, it follows that
\bea
\nono&&\Phi_{\mu}(x_{k+1})\leq\Phi_{\mu}(x_k)-\aaa_k(1-\f{L_{i_k}}{2}\aaa_k)\|\pg\|_{i_k}^2\\
&&\ \ \ \ \ \ \ \ \ \ \ \ \ \ \ \ +\aaa_{k}\nn\w{\Delta}_k,\pg\mm \label{start}
\eea
By setting $s=i_k$, $x=x_k$, $\aaa=\aaa_k$, $g_1=\nabla f_{\mu}(x_k)$, $g_2=U_sG_{\mu,k}^{(s)}$ and $s=i_k$, $x=x_k$, $\aaa=\aaa_k$, $g_1=\nabla f_{\mu}(x_k)$, $g_2=\nabla f(x_k)$ in Lemma \ref{lemma2}, the following two relations are obtained,
\bea
\nono&&\|\pfm-\pg\|_{i_k}\leq\|\nabla_{i_k}f_{\mu}(x_k)-U_{i_k}\G\|_{i_k}, \\
\nono&&\|\pfm-\pf\|_{i_k}\leq\|\nabla_{i_k} f_{\mu}(x_k)-\nabla_{i_k}f(x_k)\|_{i_k}.
\eea
With above two relations, the following estimate for $\|\pg\|_{i_k}^2$ in \eqref{start} holds,
\bea
\nono&&\|\pf\|_{i_k}\leq2\|\pf-\pfm\|_{i_k}^2 \\
\nono&&\ \ \ \ \ \ \ \ \ \ \ \ \ \ \ \ \ \ \ \ \ \ \ \ \ \ \ \ \ \ \ \ \ \ +2\|\pfm\|_{i_k}^2, \\
\nono&& \ \ \ \ \ \ \ \ \ \ \ \ \ \ \ \ \ \ \ \ \ \ \ \ \ \ \ \ \ \ \ \ \leq 2\|\pf-\pfm\|_{i_k}^2 \\
\nono&&  \ \ \ \ \ \ \ \ \ \ \ \ \ \ \ \ \ \ \ \ \ \ \ \ \ \ \ \ \ \ \ \ \ \ +4\|\pfm-\pg\|_{i_k}^2 \\
\nono&&  \ \ \ \ \ \ \ \ \ \ \ \ \ \ \ \ \ \ \ \ \ \ \ \ \ \ \ \ \ \ \ \ \ \ +4\|\pg\|_{i_k}^2 \\
\nono&&\ \ \ \ \ \ \ \ \ \ \ \ \ \ \ \ \ \ \ \ \ \ \ \ \ \ \ \ \ \ \ \ \leq2\|\nabla_{i_k}f_{\mu}(x_k)-\nabla_{i_k}f(x_k)\|_{i_k}^2+4\|\G-\nabla_{i_k}f_{\mu}(x_k)\|_{i_k}^2 \\
&&\ \ \ \ \ \ \ \ \ \ \ \ \ \ \ \ \ \ \ \ \ \ \ \ \ \ \ \ \ \ \ \ \ \ +4\|\pg\|_{i_k}^2. \label{est1}
\eea
On the other hand, by using Cauchy inequality and Lemma \ref{lemma2}, we have
\bea
\nono&&\nn\w{\Delta}_k,\pg\mm=\nn\w{\Delta}_k,\pg-\pfm\mm  \\
\nono&&\ \ \ \ \ \ \ \ \ \ \ \ \ \ \ \ \ \ \ \ \ \ \ \ \ \ \ \ \ \ \ \ \ \ \ \ \ \ \ +\nn\w{\Delta}_k,\pfm\mm \\
\nono&&\ \ \ \ \ \ \ \ \ \ \ \ \ \ \ \ \ \ \ \ \ \ \ \ \ \ \ \ \ \ \ \ \ \ \  \leq \|\w{\Delta}_k\|_{i_k}\cdot\|\pg-\pfm\|_{i_k} \\
&&\ \ \ \ \ \ \ \ \ \ \ \ \ \ \ \ \ \ \ \ \ \ \ \ \ \ \ \ \ \ \ \ \ \ \ \ \ \ \ +\nn\w{\Delta}_k,\pfm\mm.   \label{est2}
\eea
Substitute estimate \eqref{est1} and \eqref{est2} into \eqref{start}, it follows that
\bea
\nono&&\Phi_{\mu}(x_{k+1})\leq\Phi_{\mu}(x_k)-\aaa_k(1-\f{L_{i_k}}{2}\aaa_k)\Big[\f{1}{4}\|\pf\|_{i_k}^2\\
\nono&&\ \ \ \ \ \ \ \ \ \ \ \ \ \ \ \ \ -\f{1}{2}\|\nabla_{i_k}f_{\mu}(x_k)-\nabla_{i_k}f(x_k)\|_{i_k}^2-\|\w{\Delta}_k\|_{i_k}^2\Big]+\aaa_k\|\w{\Delta}_k\|_{i_k}^2 \\
\nono&&\ \ \ \ \ \ \ \ \ \ \ \ \ \ \ \ \ +\aaa_k\nn\w{\Delta}_k,\pfm\mm.
\eea
Sum up both sides from $k=1$ to $T$ and rearrange terms, it follows that
\bea
\nono&&\sum_{k=1}^T\aaa_k(1-\f{L_{i_k}}{2}\aaa_k)\|\pf\|_{i_k}^2 \\
\nono&&\ \ \ \leq4\big[\Phi_{\mu}(x_k)-\Phi_{\mu}^{*}\big]+4\sum_{k=1}^T\aaa_k\nn\w{\Delta}_k,\pfm\mm  \\
&&\ \ \ \ +2\sum_{k=1}^T\aaa_k(1-\f{L_{i_k}}{2}\aaa_k)\|\nabla_{i_k}f_{\mu}(x_k)-\nabla_{i_k}f(x_k)\|_{i_k}^2+8\sum_{k=1}^T\aaa_k\|\w{\Delta}_k\|_{i_k}^2. \label{start2}
\eea
Denote $\zeta_k=(i_k,\tau_k)$, by Lemma \ref{Nes Lem} (a) and the fact that $\p_s(x_k,\nabla f_{\mu}(x_k),\aaa_k)$ is measurable with respect to the history $\zeta_{[k-1]}$ when $s=1,2,...,b$, we have
\bea
\nono&&\mbb E_{\zeta_{[k-1]}}[\nn\w{\Delta}_k,\pfm\mm] \\
&&=\sum_{s=1}^b\mbb E_{\zeta_{[k-1]}}[\nn G_{\mu,k}^{(s)}(x_k)-\nabla_{s} f_{\mu}(x_k),\p_s(x_k,\nabla f_{\mu}(x_k),\aaa_k)\mm]=0.  \label{fuzhu1}
\eea
Denote $\tau_{k,t}=(\xi_{k,t},u_{k,t})$ and $\w{\Delta}_{k,t}=G_{\mu}^{(i_k)}(x_k,\tau_{k,t})-\nabla_{i_k}f_{\mu}(x_k)$, $t=1,2,...,T_k$, $H_t=\sum_{i=1}^t\w{\Delta}_{k,t}$, $t=1,2,...,T_k$, $H_0=0$, then the property of stochastic oracle implies
\bea
\nono\mbb E_{\zeta_{[k-1]}}[\nn H_{t-1},\w{\Delta}_{k,t}\mm|H_{t-1}]=0, t=1,2,...,T_k,
\eea and
\bea
\nono &&\mbb E_{\zeta_{[k-1]}}[\|H_{T_k}\|_{i_k}^2]=\mbb E_{\zeta_{[k-1]}}[\|H_{T_k-1}\|_{i_k}^2]+\mbb E_{\zeta_{[k-1]}}[\|\w{\Delta}_{k,T_k}\|_{i_k}^2]+2\mbb E_{\zeta_{[k-1]}}[\nn H_{T_k-1},\w{\Delta}_{k,T_k}\mm] \\
\nono&&\ \ \ \ \ \ \ \ \ \ \ \ \ \ \ =\mbb E_{\zeta_{[k-1]}}[\|\w{\Delta}_{k,T_k}\|_{i_k}^2]+\mbb E_{\zeta_{[k-1]}}[\|H_{T_k-1}\|_{i_k}^2]=\cdots=\sum_{i=1}^{T_k}\mbb E_{\zeta_{[k-1]}}[\|\w{\Delta}_{k,i}\|_{i_k}^2].
\eea
Then the estimate for $\|\w{\Delta}_k\|_{i_k}^2$ in \eqref{start2} is obtained as follow,
\bea
\nono&&\mbb E_{\zeta_{[k-1]}}[\|\w{\Delta}_k\|_{i_k}^2]=\mbb E_{\zeta_{[k-1]}}[\|\f{1}{T_k}\sum_{t=1}^{T_k}\w{\Delta}_{k,t}\|_{i_k}^2]=\f{1}{T_k^2}\mbb E_{\zeta_{[k-1]}}[\|H_{T_k}\|_{i_k}^2]=\f{1}{T_k^2}\sum_{t=1}^{T_k}\mbb E_{\zeta_{[k-1]}}[\|\w{\Delta}_{k,t}\|_{i_k}^2] \\
\nono&&\ \ \ \ =\f{1}{T_k^2}\sum_{t=1}^{T_k}\sum_{s=1}^bp_s\mbb E_{\tau_{[k-1]}}[\|G_{\mu}^{(s)}(x_k,\tau_{k,t})-\nabla_{s}f_{\mu}(x_k)\|_{s}^2] \\ \nono&&\ \ \ \ \leq\f{\max_{s\in\b}p_s}{T_k^2}\sum_{t=1}^{T_k}\sum_{s=1}^b\mbb E_{\tau_{[k-1]}}[\|G_{\mu}^{(s)}(x_k,\tau_{k,t})-\nabla_{s}f_{\mu}(x_k)\|_{s}^2]\\
&& \ \ \ \ =\f{\max_{s\in\b}p_s}{T_k^2}\sum_{t=1}^{T_k}\mbb E_{\tau_{[k-1]}}[\|G_{\mu}(x_k,\tau_{k,t})-\nabla f_{\mu}(x_k)\|^2].\label{tui1}
\eea
Also note that
\bea
\nono&&\mbb E_{\tau_{[k-1]}}[\|G_{\mu}(x_k,\tau_{k,t})-\nabla f_{\mu}(x_k)\|^2]\leq2\mbb E_{\tau_{[k-1]}}[\|G_{\mu}(x_k,\tau_{k,t})\|^2]+2\mbb E_{\tau_{[k-1]}}[\|\nabla f_{\mu}(x_k)\|^2] \\
\nono&&\ \ \ \ \ \ \ \ \ \ \ \leq2\bigg[2(n+4)\big[E_{\tau_{[k-1]}}[\|\nabla f(x_k)\|^2]+\sigma^2\big]+\f{\mu^2}{2}L_f^2(n+6)^3\bigg] \\
\nono&&\ \ \ \ \ \ \ \ \ \ \ \ +2\bigg[2\mbb E_{\tau_{[k-1]}}[\|\nabla f_{\mu}(x_k)-\nabla f(x_k)\|^2]+2\mbb E_{\tau_{[k-1]}}[\|\nabla f(x_k)\|^2]\bigg] \\
\nono&&\ \ \ \ \ \ \ \ \ \ \ \leq 2\big[2(n+4)(M^2+\sigma^2)+\f{\mu^2}{2}L_f^2(n+6)^3\big]+2\big[\f{\mu^2}{2}L_f^2(n+3)^3+2\mu^2\big] \\
&&\ \ \ \ \ \ \ \ \ \ \ \leq4(n+4)\big[2M^2+\sigma^2+\mu^2L_f^2(n+4)^2\big]:=\w{\sigma}^2.  \label{sigma}
\eea
in which the inequality follows from Lemma \ref{Nes Lem} (c). Then \eqref{tui1},\eqref{sigma} implies
\bea
\mbb E_{\zeta_{[k-1]}}[\|\w{\Delta}_k\|_{i_k}^2]\leq\f{(\max_{s\in\b}p_s)\w{\sigma}^2}{T_k}. \label{fuzhu2}
\eea
Note that
\bea
\nono&&\mbb E_{\zeta_{[k-1]}}\big[(1-\f{L_{i_k}}{2}\aaa_k)\|\pf\|_{i_k}^2\big] \\
\nono&& =\sum_{s=1}^bp_s (1-\f{L_{s}}{2}\aaa_k)\mbb E_{\tau_{[k-1]}}\|\p_s(x_k,\nabla f(x_k),\aaa_k)\|_{s}^2\big] \\
\nono&&\geq\min_{s\in\b}\big\{p_s (1-\f{L_{s}}{2}\aaa_k)\big\}\bigg(\sum_{s=1}^b\mbb E_{\tau_{[k-1]}}[\|\p_s(x_k,\nabla f(x_k),\aaa_k)\|_{s}^2]\bigg) \\
&&=\min_{s\in\b}\big\{p_s (1-\f{L_{s}}{2}\aaa_k)\big\}\mbb E_{\tau_{[k-1]}}\big[\|\p(x_k,\nabla f(x_k),\aaa_k)\|^2\big]   \label{fuzhu3}
\eea
and
\bea
\nono&&\mbb E_{\zeta_{[k-1]}}\big[(1-\f{L_{i_k}}{2}\aaa_k)\|\nabla_{i_k}f_{\mu}(x_k)-\nabla_{i_k}f(x_k)\|_{i_k}^2\big] \\
\nono&&=\sum_{s=1}^bp_s(1-\f{L_s}{2}\aaa_k)\mbb E_{\tau_{[k-1]}}\big[\|\nabla_{s}f_{\mu}(x_k)-\nabla_{s}f(x_k)\|_{s}^2\big] \\
\nono&&\leq \max_{s\in\b}\big\{p_s(1-\f{L_s}{2}\aaa_k)\big\}\cdot\sum_{s=1}^b\mbb E_{\tau_{[k-1]}}\big[\|\nabla_{s}f_{\mu}(x_k)-\nabla_{s}f(x_k)\|_{s}^2\big]\\
\nono&&=\max_{s\in\b}\big\{p_s(1-\f{L_s}{2}\aaa_k)\big\}\mbb E_{\tau_{[k-1]}}\big[\|\nabla f_{\mu}(x_k)-\nabla f(x_k)\|^2\big]\\
&&\leq \f{\mu^2}{4}L_f^2(n+3)^3\max_{s\in\b}\big\{p_s(1-\f{L_s}{2}\aaa_k)\big\}.   \label{fuzhu4}
\eea
Combine \eqref{fuzhu1}, \eqref{fuzhu2}, \eqref{fuzhu3}, \eqref{fuzhu4}, and take expectation on both sides of \eqref{start2} with respect to $\zeta_{[T]}$, it follows that
\bea
\nono&&\sum_{k=1}^T\aaa_k\min_{s\in\b}\big\{p_s (1-\f{L_{s}}{2}\aaa_k)\big\}\mbb E_{\zeta_{[T]}}\big[\|\p(x_k,\nabla f(x_k),\aaa_k)\|^2\big]\\
\nono&&\leq4[\Phi_{\mu}(x_1)-\Phi_{\mu}^{*}]+8\max_{s\in\b}p_s\w{\sigma}^2\sum_{k=1}^T\f{\aaa_k}{T_k}\\
\nono&&\ \ \ +\f{\mu^2}{2}L_f^2(n+3)^3\sum_{k=1}^T\aaa_k\max_{s\in\b}\big\{p_s(1-\f{L_s}{2}\aaa_k)\big\}.
\eea
Divide both sides of the above inequality by $\sum_{k=1}^T\aaa_k\min_{s\in\b}\big\{p_s (1-\f{L_{s}}{2}\aaa_k)\big\}$ and note that $\sum_{k=1}^T\aaa_k\max_{s\in\b}\big\{p_s (1-\f{L_{s}}{2}\aaa_k)\big\}/\sum_{k=1}^T\aaa_k\min_{s\in\b}\big\{p_s (1-\f{L_{s}}{2}\aaa_k)\big\}\geq1$, the desired result is obtained.
\end{proof}

\begin{cor}\label{cor2}
Under Assumptions of Theorem \ref{general main}. Denote $\hat{L}=\max_{s\in\b}L_s$, suppose that in the ZS-BMD algorithm the stepsizes $\aaa_k=1/\hat{L}$, the random variables $\{i_k\}$ are uniformly distributed ($p_1=p_2=\cdots=p_b=1/b$). Also suppose that $G_{\mu,k}^{(i_k)}$ is computed as \eqref{minig} and the batch sample size of each step $T_k=T'$, then we have
\bea
\mbb E[\|\p(x_R,\nabla f(x_R),\aaa_R)\|^2]\leq\f{8b\hat{L}^2D_{\Phi}^2+8\mu^2L_f^2nb}{T}+\f{16\w{\sigma}^2b}{T'}+\f{\mu^2}{2}L_f^2(n+3)^3,\label{cor1}
\eea
in which $D_{\Phi}=[(\Phi(x_1)-\Phi^{*})/\hat{L}]^{\f{1}{2}}$, $\w{\sigma}^2$ is defined as in \eqref{sigma}, and the expectation is taken w.r.t. $R$, $i_{[T]}$, $\xi_{[T]}$ and $u_{[T]}$.
\end{cor}
\begin{proof}
Note that
\be
\nono|\Phi_{\mu}(x)-\Phi_{\mu}^{*}-(\Phi(x)-\Phi^{*})|\leq \mu^2L_fn,
\ee
and
\be
\nono\sum_{k=1}^T\aaa_k\min_{s\in\b} p_s(1-\f{L_s}{2}\aaa_k)\geq \sum_{k=1}^T\f{1}{\hat{L}}\cdot\f{1}{2b}=\f{T}{2b\hat{L}}.
\ee
Then, \eqref{cor1} follows by substituting the two inequalities into \eqref{generalin}.
\end{proof}
\begin{rmk}\label{comrmk1}
Corollary \ref{cor2} also indicates that, if the smoothing parameter is taken as $\mu= O(1/(n+4)^{\f{3}{2}}\sqrt{T})$, the batch sample size $T'$ is taken as $T'=O\big((n+4)T\big)$, in order to get an $\epsilon$-stationary point of composite problem \eqref{pro2}, the total number of calls of ZS-BMD to the SZO can be bounded by
$\o(b^2n/\epsilon^2)$, which demonstrates a linear dependence on the dimension and a quadratic dependence on the block index.
\end{rmk}
\begin{cor}
Under conditions of Corollary \ref{cor2}, let $\w{T}$ be a fixed total number of calls to the SZO, suppose the smoothing parameter $\mu$ satisfies
\be
\mu\leq \f{D_{\Phi}}{n+4}\sqrt{\f{1}{\w{T}}}, \label{smoothmu}
\ee
and the number of calls to SZO at each iteration step of the ZS-BMD method is
\be
T'=\bigg\lceil\min\bigg\{\max\bigg\{\f{\sqrt{(n+4)(2M^2+\sigma^2)\w{T}}}{\w{L}\w{D}},n+4\bigg\},\w{T}\bigg\}\bigg\rceil  \label{batchsize}
\ee
for some $\w{D}$, in which $\w{L}=\max\big\{L_f,\hat{L}\big\}$. Then the generalized Bregman projected gradient satisfies
\be
\nono\f{1}{\w{L}b}\mbb E\big\{\|\p(x_R,\nabla f(x_R),\aaa_R)\|^2\big\}\leq\textbf{B}_{\w{T}},
\ee
in which
\be
\textbf{B}_{\w{T}}=\f{64\sqrt{(n+4)(2M^2+\sigma^2)}}{\sqrt{\w{T}}}\bigg(\w{D}\gamma_1+\f{D_{\Phi}^2}{\w{D}}\bigg)+\f{(64\gamma_2+33)\w{L}D_{\Phi}^2(n+4)}{\w{T}},\label{dab}
\ee
and
\be
\nono\gamma_1=\max\bigg\{\f{\sqrt{(n+4)(2M^2+\sigma^2)}}{\w{L}\w{D}\sqrt{\w{T}}},1\bigg\}, \ \ \gamma_2=\max\bigg\{\f{n+4}{\w{T}},1\bigg\}.
\ee
\end{cor}
\begin{proof}
Note that $T=\lfloor\w{T}/T'\rfloor$, it's obvious that $T\geq\w{T}/2T'$ and $T'\geq1$. Also note that $\mu$ satisfies \eqref{smoothmu}, then from Corollary \ref{cor2}, we have
\bea
\nono&&\eqref{cor1}\leq\f{16\w{L}^2D_{\Phi}^2bT'}{\w{T}}\bigg(1+\f{1}{(n+4)\w{T}}\bigg)+\f{64(n+4)(2M^2+\sigma^2)b}{T'}+\f{64(n+4)L_f^2D_{\Phi}^2b}{T'\w{T}}+\f{L_f^2D_{\Phi}^2(n+4)b}{2\w{T}} \\
\nono&&\ \ \ \ \ \ \ \ \ \leq\f{16\w{L}^2D_{\Phi}^2b}{\w{T}}\bigg[\f{\sqrt{(n+4)(2M^2+\sigma^2)\w{T}}}{\w{L}\w{D}\gamma_1}+\f{n+4}{\gamma_2}\bigg]+\f{16\w{L}^2D_{\Phi}^2b}{(n+4)\w{T}} \\
\nono&&\ \ \ \ \ \ \ \ \ \ \ \ +\f{64\gamma_1\w{L}\w{D}b\sqrt{(n+4)(2M^2+\sigma^2)}}{\sqrt{\w{T}}}+\f{64\gamma_2\w{L}^2D_{\Phi}^2(n+4)b}{\w{T}}+\f{\w{L}^2D_{\Phi}^2(n+4)b}{2\w{T}} \\
\nono&&\ \ \ \ \ \ \ \ \  \leq \f{16\w{L}D_{\Phi}^2b\sqrt{(n+4)(2M^2+\sigma^2)}}{\w{D}\sqrt{\w{T}}}+\f{16\w{L}^2D_{\Phi}^2b(n+4)}{\w{T}}+\f{16\w{L}^2D_{\Phi}^2b}{\w{T}} \\
\nono&&\ \ \ \ \ \ \ \ \ \ \ \  +\f{64\gamma_1\w{L}\w{D}b\sqrt{(n+4)(2M^2+\sigma^2)}}{\sqrt{\w{T}}}+\f{64\gamma_2\w{L}^2D_{\Phi}^2(n+4)b}{\w{T}}+\f{\w{L}^2D_{\Phi}^2(n+4)b}{2\w{T}} .
\eea
After rearranging terms, the desired result is obtained.
\end{proof}

\subsection{Two-phase optimization scheme for ZS-BMD}
We propose to establish the two-phase ZS-BMD (2-ZS-BMD) method in this section.
\begin{framed}
\noindent\textbf{Two-phase ZS-BMD:}

\noi\textbf{Input}: Initial point $x_1\in X_1\times X_2\times\cdots\times X_b$, generate the random variables $\{i_k\}$ that are uniformly distributed ($p_1=p_2=\cdots=p_b=1/b$), number of runs $S$, total $\w{T}$ of calls to SZO in each run of ZS-BMD, sample size $\t$ in the post-optimization phase.

\noi\textbf{Optimization phase}:

\noindent For $i=1,2,...,S$, call the ZS-BMD algorithm with input $x_1$, batch sizes $T_k=T'$ which is as in \eqref{batchsize}, iteration limit $T=\lfloor\w{T}/T'\rfloor$, stepsizes $\aaa_k=1/\hat{L}$, $k=1,2,...,T$, $\hat{L}$ is as in Corollary \ref{cor2}, smoothing parameter $\mu$ satisfying \eqref{smoothmu} and probability probability function $P_R$ in \eqref{generalp}. Output $\bar{x}_i=x_{R_i}$, $i=1,2,...,S$.

\noi\textbf{Post-optimization phase}:

\noindent Select a solution $\bar{x}_{*}$ from the candidate list $\{\bar{x}_1,\bar{x}_2,...,\bar{x}_S\}$ such that
\be
\nono \bar{x}_{*}=\arg\min_{i=1,2,...,S}\big\|\g_{G}(\bar{x}_{i})\big\|, \ \g_{G}(\bar{x}_{i})=\p(\bar{x}_i,G_{\mu,\t}(\bar{x}_i),\aaa_{R_i}),
\ee
in which $G_{\mu,\t}(x)=\f{1}{\t}\sum_{k=1}^{\t}G_{\mu}(x,\xi_k,u_k)$.

\noi\textbf{Output $\bar{x}_{*}$.}
\end{framed}
To  obtain main results for 2-ZS-BMD, the following basic lemma of martingale difference sequence is needed \cite{zero first}.
\begin{lem}\label{martingale}
For a polish space $\Omega$ with Borel probability measure $\mu$ and an increasing $\sigma$-subalgebras $\mathcal{F}_0=\{\Omega,\emptyset\}\subseteq\mathcal{F}_1\subseteq\mathcal{F}_2\subseteq\cdots$ of Borel $\sigma$-algebra of $\Omega$. Suppose that $\{\psi_i\}\in \mbb R^n$, $i=1,2,...\infty$ is a martingale sequence of Borel functions on $\Omega$ satisfying $\psi_i$ is $\mathcal{F}_i$ measurable and $\mbb E[\psi_i|\mathcal{F}_{i-1}]=0$. If $\mbb E[\|\psi_i\|^2]\leq\nu_i^2$ for any $i\geq1$, then for any $N\geq1$ and $\lambda\geq0$, the following large-deviation property of martingales holds:
\be
\nono Prob\bigg\{\big\|\sum_{i=1}^N\psi_i\big\|^2\geq\lambda\sum_{i=1}^N\sigma_i^2\bigg\}\leq\f{1}{\lambda}.
\ee
\end{lem}

We denote $\g_f(\bar{x}_i)=\p(\bar{x}_i,\nabla f(\bar{x}_i),\aaa_{R_i})$ and $\g_{f_{\mu}}(\bar{x}_i)=\p(\bar{x}_i,\nabla f_{\mu}(\bar{x}_i),\aaa_{R_i})$ in which $\p$ is defined as in last section. Then $\g_f(\bar{x}_i)$ denotes the generalized gradient of function $f$ at $\bar{x}_i$.
\begin{thm}\label{promain1}
Under Assumptions \ref{ass1}, \ref{ass2}, \ref{ass3}. Let $\textbf{B}_{\w{T}}$ be defined as in \eqref{dab}, $b$ is the total block number as before. Then the 2-ZS-BMD method for constrained composite optimization problem \eqref{pro2} has the following probability estimate:
\bea
\nono&&Prob\bigg\{\big\|\g_f(\bar{x}_{*})\big\|^2\geq16b\w{L}\textbf{B}_{\w{T}}+\f{3D_{\Phi}^2L_f^2b(n+4)}{\w{T}}\\
&&\ \ \ \ \ \ \ \ \ \ \ +\f{32(n+4)\lambda}{\t}\bigg[2M^2+\sigma^2+\f{D_{\Phi}^2L_f^2b}{\w{T}}\bigg]\bigg\}\leq\f{S+1}{\lambda}+2^{-S}, \ \forall\lambda>0.
\eea
\end{thm}
\begin{proof}
Start from the definition of $\bar{x}_{*}$ in 2-ZS-BMD and the notations of $\g_f$ and $\g_{f_{\mu}}$, we have
\bea
\nono&&\big\|\g_G(\bar{x}_{*})\big\|^2=\min_{i=1,2,...,S}\big\|\g_G(\bar{x}_i)\big\|^2 \\
\nono&&=\min_{i=1,2,...,S}\big\|\g_G(\bar{x}_i)-\g_{f_{\mu}}(\bar{x}_i)+\g_{f_{\mu}}(\bar{x}_i)\big\|^2 \\
\nono&&\leq\min_{i=1,2,...,S}2\big[\big\|\g_G(\bar{x}_i)-\g_{f_{\mu}}(\bar{x}_i)\big\|^2+\big\|\g_{f_{\mu}}(\bar{x}_i)\big\|^2\big] \\
\nono&&\leq\min_{i=1,2,...,S}2\big[\big\|\g_G(\bar{x}_i)-\g_{f_{\mu}}(\bar{x}_i)\big\|^2+2\big\|\g_{f_{\mu}}(\bar{x}_i)-\g_f(\bar{x}_i)\big\|^2+2\big\|\g_f(\bar{x}_i)\big\|^2\big] \\ \nono&&\leq4\min_{i=1,2,...,S}\big\|\g_f(\bar{x}_i)\big\|^2+2\max_{i=1,2,...,S}\big\|\g_G(\bar{x}_i)-\g_{f_{\mu}}(\bar{x}_i)\big\|^2\\
&&\ \ \ +4\max_{i=1,2,...,S}\big\|\g_{f_{\mu}}(\bar{x}_i)-\g_f(\bar{x}_i)\big\|^2.\label{ine1}
\eea
Then, it follows that
\bea
\nono&&\big\|\g_f(\bar{x}_{*})\big\|^2\leq2\big\|\g_G(\bar{x}_{*})\big\|^2+2\big\|\g_G(\bar{x}_{*})-\g_f(\bar{x}_{*})\big\|^2  \\
\nono&&\leq2\big\|\g_G(\bar{x}_{*})\big\|^2+4\big\|\g_G(\bar{x}_{*})-\g_{f_{\mu}}(\bar{x}_{*})\big\|^2+4\big\|\g_{f_{\mu}}(\bar{x}_{*})-\g_f(\bar{x}_{*})\big\|^2 \\
\nono&& \leq 8\min_{i=1,2,...,S}\big\|\g_f(\bar{x}_i)\big\|^2+4\max_{i=1,2,...,S}\big\|\g_G(\bar{x}_i)-\g_{f_{\mu}}(\bar{x}_i)\big\|^2\\
\nono&&\ \ +8\max_{i=1,2,...,S}\|\g_{f_{\mu}}(\bar{x}_i)-\g_f(\bar{x}_i)\|^2+4\big\|\g_G(\bar{x}_*)-g_{f_{\mu}}(\bar{x}_{*})\big\|^2 \\
\nono&&\ \ +4\big\|\g_{f_{\mu}}(\bar{x}_{*})-\g_f(\bar{x}_{*})\big\|^2 \\
\nono&&\leq 8\min_{i=1,2,...,S}\big\|\g_f(\bar{x}_i)\big\|^2+4\max_{i=1,2,...,S}\big\|\g_G(\bar{x}_i)-\g_{f_{\mu}}(\bar{x}_i)\big\|^2 \\
&&\ \ +4\big\|\g_G(\bar{x}_{*})-\g_{f_{\mu}}(\bar{x}_{*})\big\|^2+12\max_{i=1,2,...,S}\big\|\g_{f_{\mu}}(\bar{x}_i)-\g_f(\bar{x}_i)\big\|^2, \label{two1}
\eea
in which the second inequality follows by substituting \eqref{ine1} into the first term of the right hand side. Use the fact that $\bar{x}_i$, $i=1,2,...,S$ are independent and Markov inequality, it can be obtained that
\be
Prob\bigg\{\min_{i=1,2,...,S}\big\|\g_f(\bar{x}_i)\big\|^2\geq 2b\w{L}\textbf{B}_{\w{T}}\bigg\}=\prod_{i=1}^SProb\bigg\{\big\|\g_f(\bar{x}_i)\big\|^2\geq 2b\w{L}\textbf{B}_{\w{T}}\bigg\}\leq\f{1}{2^S}. \label{two2}
\ee
Denote $\frkd_{i,k}=G_{\mu}(\bar{x}_i,\xi_k,u_k)-\nabla f_{\mu}(\bar{x}_i)$, $k=1,2,...,\t$, it follows from previous calculation and \eqref{smoothmu} that
\bea
\mbb E[\|\frkd_{i,k}\|^2]\leq 4(n+4)\bigg(2M^2+\sigma^2+\f{D_{\Phi}^2L_f^2b}{\w{T}}\bigg)=:\c_{\w{T}}. \label{dac}
\eea
Since $G_{\mu,\t}(\bar{x}_i)-\nabla f_{\mu}(\bar{x}_i)=\sum_{k=1}^{\t}\frkd_{i,k}/\t$, Lemma \ref{lemma2} , \eqref{dac}, and Lemma \ref{martingale} imply that, for any $i=1,2,...,S$,
\bea
\nono&&Prob\bigg\{\big\|\g_G(\bar{x}_i)-\g_{f_{\mu}}(\bar{x}_i)\big\|^2\geq\f{\lambda\c_{\w{T}}}{\t}\bigg\}\leq Prob\bigg\{\big\|G_{\mu,\t}(\bar{x}_i)-\nabla f_{\mu}(\bar{x}_i)\big\|^2\geq\f{\lambda\c_{\w{T}}}{\t}\bigg\}  \\
\nono&&\ \ \ \ \ \ \ \ \ \ \ \ \ \ \ \ \ \ \ \  \ \ \ \ \ \ \ \ \ \ \  \ \ \ \ \ \ \ \ \ \ \ \ \ \ \ \ =Prob\bigg\{\big\|\sum_{k=1}^{\t}\frkd_{i,k}\big\|^2\geq\lambda\t\c_{\w{T}}\bigg\}\leq\f{1}{\lambda}, \forall \lambda>0,
\eea
in which the first inequality follows from the block non-expansion property and the second inequality follows from Lemma \ref{martingale}. Then it follows that
\be
Prob\bigg\{\max_{i=1,2,...,S}\big\|\g_G(\bar{x}_i)-\g_{f_{\mu}}(\bar{x}_i)\big\|^2\geq\f{\lambda\c_{\w{T}}}{\t}\bigg\}\leq\f{S}{\lambda},\forall \lambda>0. \label{two3}
\ee
and
\be
Prob\bigg\{\big\|\g_G(\bar{x}_{*})-\g_{f_{\mu}}(\bar{x}_{*})\big\|^2\geq\f{\lambda\c_{\w{T}}}{\t}\bigg\}\leq\f{1}{\lambda},\forall \lambda>0.\label{two4}
\ee
Also note that
\be
\nono\big\|\g_{f_{\mu}}(\bar{x}_i)-\g_f(\bar{x}_i)\big\|^2\leq\big\|\nabla f_{\mu}(\bar{x}_i)-\nabla f(\bar{x}_i)\|^2\leq\f{\mu^2}{4}L_f^2(n+3)^3\leq\f{D_{\Phi}^2L_f^2b(n+4)}{4\w{T}},
\ee
take $\max$ over index $i=1,2,...,S$ on both sides, it follows that
\be
\max_{i=1,2,...,S}\big\|\g_{f_{\mu}}(\bar{x}_i)-\g_f(\bar{x}_i)\big\|^2\leq\f{D_{\Phi}^2L_f^2b(n+4)}{4\w{T}}. \label{inef}
\ee
Now combine \eqref{two1}, \eqref{two2}, \eqref{two3}, \eqref{two4} and \eqref{inef}, the desired result then follows.
\end{proof}

\begin{cor}\label{compe}
Under assumptions of Theorem \ref{promain1}, let $\epsilon>0$ and $\Lambda\in (0,1)$, suppose the parameters $(S,\w{T},\t)$ are selected as
\bea
\nono&&S(\Lambda)=\big\lceil \log_2(\f{2}{\Lambda})\big\rceil,\label{S}\\
\nono&&\w{T}(\epsilon)=\bigg\lceil\max\bigg\{n+4, \f{(n+4)(2M^2+\sigma^2)}{\w{L}^2\w{D}^2},\f{99\cdot8^2(n+4)b\w{L}^2D_{\Phi}^2}{\epsilon},\\
\nono&&\ \ \ \ \ \ \ \ \ \  \bigg[\f{66\cdot32b\sqrt{(n+4)(2M^2+\sigma^2)}}{\epsilon}\bigg(\w{D}+\f{D_{\Phi}^2}{\w{D}}\bigg)\bigg]^2\bigg\} \bigg\rceil \label{wt}\\
\nono&&\t(\epsilon,\Lambda)=\bigg\lceil\f{32(n+4)\cdot2(S+1)}{\Lambda}\max\bigg\{1,\f{16(2M^2+\sigma^2)}{\epsilon}\bigg\} \bigg\rceil  \label{gt}
\eea
then, the 2-ZS-BMD computes an $(\epsilon,\Lambda)$-solution of problem \eqref{pro2} after taking at most total number of calls
\be
S(\Lambda)\big[\w{T}(\epsilon)+\t(\epsilon,\Lambda)\big] \label{comp}
\ee
to SZO.
\end{cor}
\begin{proof}
It's obvious that the total number of calls to SZO in 2-ZS-BMD algorithm is bounded by $S(\Lambda)\big[\w{T}(\epsilon)+\t(\epsilon,\Lambda)\big]$. It's sufficient to show that $\bar{x}_{*}$ is indeed the $(\epsilon,\Lambda)$-solution of the problem \eqref{pro2}. By the selection rule of $\w{T}(\epsilon)$ and the definition of $\textbf{B}_{\w{T}}$ as in \eqref{dab}, it can be obtained that
\bea
\nono&&\textbf{B}_{\w{T}}\leq\f{99\w{L}D_{\Phi}^2(n+4)}{\w{T}}+\f{66\sqrt{(n+4)(2M^2+\sigma^2)}}{\sqrt{\w{T}}}\bigg(\w{D}+\f{D_{\Phi}^2}{\w{D}}\bigg)  \\
\nono&&\ \ \ \ \ \leq\f{\epsilon}{64\w{L}b}+\f{\epsilon}{32\w{L}b}=\f{3\epsilon}{64\w{L}b},
\eea
which also shows that
\be
16b\w{L}\textbf{B}_{\w{T}}+\f{3D_{\Phi}^2L_f^2b(n+4)}{\w{T}}\leq\f{3\epsilon}{4}+\f{\epsilon}{2112}<\f{7\epsilon}{8}.\label{e1}
\ee
On the other hand, by setting $\lambda=2(S+1)/\Lambda$ and using the selection rule of $\t(\eee,\Lambda)$, we have
\be
\f{32(n+4)\lambda}{\t}\bigg[\f{D_{\Phi}^2L_f^2b}{\w{T}}+2M^2+\sigma^2\bigg]\leq\f{\epsilon}{99\times8^2}+\f{\epsilon}{16}<\f{\epsilon}{8}.\label{e2}
\ee
By combining \eqref{e1} and \eqref{e2}, and noting the selection rule of $S(\Lambda)$, we obtain
\bea
\nono &&Prob\big\{\|\g_f(\bar{x}_{*})\|^2\geq\eee\big\}\leq Prob\bigg\{\|\g_f(\bar{x}_{*})\|^2\geq16b\w{L}\textbf{B}_{\w{T}(\eee)}+\f{3D_{\Phi}^2L_f^2b(n+4)}{\w{T}(\eee)}  \\
\nono&&\ \ \ \ \ \ \ \ \ \ \ \ \ \ \ \ \ \ \ \ \ \ \ \ \ \ \ \ \ \ \ \ \ +\f{32(n+4)\lambda}{\t(\eee,\Lambda)}\bigg[2M^2+\sigma^2+\f{D_{\Phi}^2L_f^2b}{\w{T}(\eee)}\bigg] \bigg\} \\
\nono&&\ \ \ \ \ \ \ \ \ \ \ \ \ \ \ \ \ \ \ \ \ \ \ \ \ \ \ \ \ \ \ \ \ \leq \f{S+1}{\lambda}+2^{-S}\leq \f{\Lambda}{2}+\f{\Lambda}{2}=\Lambda,
\eea
which finishes the proof.
\end{proof}

In view of Corollary \ref{compe}, the complexity bound in \eqref{comp} of 2-ZS-BMD algorithm for finding an $(\eee,\Lambda)$-solution of problem \eqref{pro2} for generalized gradient, can be bounded by
\be
\o\Bigg\{\f{bn(L_f+\hat{L})^2D_{\Phi}^2}{\eee}\log_2\bigg(\f{1}{\Lambda}\bigg)+\f{b^2n(M^2+\sigma^2)}{\eee^2}\bigg(\w{D}+\f{D_{\Phi}^2}{\w{D}}\bigg)^2\log_2\bigg(\f{1}{\Lambda}\bigg)+\f{(M^2+\sigma^2)n}{\Lambda\eee}\log_2^2\bigg(\f{1}{\Lambda}\bigg)\Bigg\}\label{complexity}
\ee
In \eqref{complexity}, index $b$ plays an important role in the analysis of complexity for 2-ZS-BMD method. In contrast to the existing work like \cite{zero first}, \cite{compotp} with no block decomposition and block projection technique involved, the performance of $b$ and $b^2$ in \eqref{complexity} shows the first and second order dependence of complexity on $b$. Also, a linear dependency of complexity on dimension $n$ is obtained. Appearance of $n$ comes from the fact that SZO is used instead of the stochastic first oracle that the existing stochastic BCD methods considered. Meanwhile, this is the first complexity result of block coordinate type method for searching for $(\eee,\Lambda)$-solution of nonconvex composite optimization problem via a two-phase procedure.
\section{Constrained composite optimization: ZS-BCCG}
In this section, we develop ZS-BCCG method to solve constrained nonconvex (composite) optimization problem. In each iteration of the proposed algorithm, only one random block implements a zeroth-order stochastic conditional gradient descent procedure. Thus, the proposed algorithm replaces the projection over the whole space $X$ of projected gradient descent type algorithm by a zeroth-order linear programming over only one block $X_s$. From the computational perspective, it significantly reduces the projection cost by transforming the problem to a linear subproblem, which is often easy to solve. Meanwhile, direct information on gradient of the objective function is also not needed.

\subsection{For nonconvex smooth objective function}

\begin{framed}
\noindent\textbf{ZS-BCCG method:}

\noindent\textbf{Input}: Inital point $z_1\in X_1\times X_2\times\cdots\times X_b$, smoothing parameter $\mu$, iteration limit $T$, stepsizes $\aaa_k$, $k\geq1$, probability mass function $P_R(\cdot)$ supported on $\{1,2,...,T\}$, probabilities $p_s\in [0,1], s=1,2,...,b$, s.t. $\sum_{s=1}^bp_s=1$.  batch sizes $T_k$ with $T_k>0$, $k\geq1$, probability mass function $P_R(\cdot)$ supported on $\{1,2,...,T\}$, probabilities $p_s\in [0,1], s=1,2,...,b$, s.t. $\sum_{s=1}^bp_s=1$.

\noindent\textbf{Step} 0: Generate a random variable $i_k$ according to \eqref{distributep}
and let $R$ be a random variable with probability mass function $P_R$.

\noindent\textbf{Step} k=1,2,...,R-1: Generate $u_k=[u_{k,1},...,u_{k,T_k}]$, in which $u_{k,j}\sim N(0,I_n)$ and call the stochastic oracle to compute the $i_k$th block average stochastic gradient $\bar{G}_{\mu,k}^{(i_k)}$ by
\begin{equation}
\bar{G}_{\mu,k}^{(i_k)}=\f{1}{T_k}\sum_{t=1}^{T_k}U_{i_k}^TG_{\mu}(z_k,\xi_{k,t},u_{k,t}). \label{minig}
\end{equation}

\noindent \underline{Update $z_{k}$} by: If $s\neq i_k$, set $z_{k+1}^{(s)}=z_k^{(s)}$; else run
\bea
x_{k+1}^{(i_k)}&=& \arg\min_{y\in X_{i_k}}\nn\bar{G}_{\mu,k}^{(i_k)},y\mm \label{con1}\\
z_{k+1}^{(i_k)}&=&(1-\aaa_k)z_k^{(i_k)}+\aaa_k x_{k+1}^{(i_k)}. \label{con2}
\eea

\noindent\textbf{Output} $z_R$.
\end{framed}
We need several assumptions for part of results in this session.

\begin{ass}\label{ass4}
Each block $X_s$ is bounded and there exists $D_{X_s}$ such that $\max_{x,y\in X_s}\|x-y\|\leq D_{X_s}$ holds for $s=1,2,...,b$.
\end{ass}
We extend the well-known Frank-Wolfe gap(FW-gap) given by
\be
\nono g_X^k=g_X(z_k)=\nn\nabla f(z_k),z_k-\w{x}_{k+1}\mm, \ \w{x}_{k+1}=\arg\min_{y\in X}\nn\nabla f(z_k),y\mm,
\ee
to block coordinate setting. Define the $s$-block FW-gap by
\be
\nono g_s^k=g_s(z_k)=\nn\nabla_s f(z_k),z_k^{(s)}-\hat{x}_{k+1}^{(s)}\mm, \ \hat{x}_{k+1}^{(s)}=\arg\min_{y\in X_s}\nn\nabla_s f(z_k),y\mm, \ s=1,2,...,b.
\ee

\begin{thm}\label{conthm1}
Under Assumptions \ref{ass1}, \ref{ass2}, \ref{ass3}, \ref{ass4}. Let $\{z_k\}_{k\geq1}$ be generated by ZS-BCCG. The probability mass function $P_R$ is chosen such that $P_R(k)=Prob\{R=k\}=\aaa_k/\sum_{k=1}^T\aaa_k$. Then, for problem \eqref{pro1}, we have
\be
\mbb E[g_X^R]\leq\f{f(z_1)-f^{*}+(\sum_{s=1}^bp_sL_sD_{X_s})\sum_{k=1}^T\aaa_k^2+\max_{s\in\b}\{p_s/L_s\}\sum_{k=1}^T\bigg[\f{\w{\sigma}^2}{T_k}+\f{\mu^2}{4}L_f^2(n+3)^3\bigg]}{(\min_sp_s)\sum_{k=1}^T\aaa_k},
\ee
in which $\w{\sigma}^2$ is as in \eqref{sigma}.
\end{thm}
\begin{proof}
Set $\bar{\Delta}_k=\bar{G}_{\mu,k}^{(i_k)}-\nabla_{i_k}f(z_k)$, use the block Lipschitz property of $f$, we have
\bea
\nono&&f(z_{k+1})\leq f(z_k)+\nn\nabla_{i_k}f(z_k),z_{k+1}^{(i_k)}-z_k^{(i_k)}\mm+\f{L_{i_k}}{2}\|z_{k+1}^{(i_k)}-z_k^{(i_k)}\|_{i_k}^2  \\
\nono&&=f(z_k)+\aaa_k\nn\bar{G}_{\mu,k}^{(i_k)}-\bar{\Delta}_k,x_{k+1}^{(i_k)}-z_k^{(i_k)}\mm+\f{L_{i_k}}{2}\aaa_k^2\|x_{k+1}^{(i_k)}-z_k^{(i_k)}\|_{i_k}^2\\
\nono&&\leq f(z_k)+\aaa_k\nn\bar{G}_{\mu,k}^{(i_k)},\hat{x}_{k+1}^{(i_k)}-z_k^{(i_k)}\mm-\aaa_k\nn\bar{\Delta}_k,x_{k+1}^{(i_k)}-z_k^{(i_k)}\mm+\f{L_{i_k}}{2}\aaa_k^2\|x_{k+1}^{(i_k)}-z_k^{(i_k)}\|_{i_k}^2\\
\nono&&=f(z_k)-\aaa_kg_{i_k}^k+\aaa_k\nn\bar{\Delta}_k,\hat{x}_{k+1}^{(i_k)}-z_k^{(i_k)}\mm-\aaa_k\nn\bar{\Delta}_k,x_{k+1}^{(i_k)}-z_k^{(i_k)}\mm+\f{L_{i_k}}{2}\aaa_k^2\|x_{k+1}^{(i_k)}-z_k^{(i_k)}\|_{i_k}^2 \\
\nono&&=f(z_k)-\aaa_kg_{i_k}^k+\aaa_k\nn\bar{\Delta}_k,\hat{x}_{k+1}^{(i_k)}-x_{k+1}^{(i_k)}\mm+\f{L_{i_k}}{2}\aaa_k^2\|x_{k+1}^{(i_k)}-z_k^{(i_k)}\|_{i_k}^2\\
\nono&&\leq f(z_k)-\aaa_kg_{i_k}^k+\f{L_{i_k}}{2}\aaa_k^2\big[\|\hat{x}_{k+1}^{(i_k)}-x_{k+1}^{(i_k)}\|_{i_k}^2+\|x_{k+1}^{(i_k)}-z_k^{(i_k)}\|_{i_k}^2\big]+\f{1}{2L_{i_k}}\|\bar{\Delta}_k\|_{i_k}^2\\
\nono&&\leq f(z_k)-\aaa_kg_{i_k}^k+L_{i_k}D_{X_{i_k}}^2\aaa_k^2+\f{1}{2L_{i_k}}\|\bar{\Delta}_k\|_{i_k}^2.
\eea
In which the first inequality follows from \eqref{con1}, second equality follows from the definition of $g_{i_k}^k$, second inequality follows from Cauchy inequality. Sum both sides of the above inequality from $k=1$ to $k=T$, it follows that
\be
f(z_{T+1})\leq f(z_1)-\sum_{k=1}^T\aaa_kg_{i_k}^k+\sum_{k=1}^TL_{i_k}D_{X_{i_k}}\aaa_k^2+\f{1}{2}\sum_{k=1}^T\f{1}{L_{i_k}}\|\bar{\Delta}_k\|_{i_k}^2. \label{bothsides}
\ee
Denote $\tau_k=(\xi_k,u_k)$, note that,  by using the similar estimate with \eqref{fuzhu1}-\eqref{fuzhu2},
\bea
\nono&&\mbb E_{i_{[k-1]},\tau_{[k-1]}}[\f{1}{L_{i_k}}\|\bar{\Delta}_k\|_{i_k}^2]\leq2\mbb E_{i_{[k-1]},\tau_{[k-1]}}[\f{1}{L_{i_k}}\|\bar{G}_{\mu,k}^{(i_k)}-\nabla_{i_k}f_{\mu}(z_k)\|_{i_k}^2]\\
\nono&&\ \ \ \ \ \ \ \ \ \ \ \ \ \ \ \ \ \ \ \ \ \ \ \ \ \ \ \  \ \ \ \ \ \ +2\mbb E_{i_{[k-1]},\tau_{[k-1]}}[\f{1}{L_{i_k}}\|\nabla_{i_k}f_{\mu}(z_k)-\nabla_{i_k}f(z_k)\|_{i_k}^2] \\
&&\ \ \ \ \ \ \ \ \ \ \ \ \ \ \ \ \ \ \ \ \ \ \ \ \ \ \ \  \ \ \ \ \leq \max_{s\in\b}\{\f{p_s}{L_s}\}\big[\f{2\w{\sigma}^2}{T_k}+\f{\mu^2}{2}L_f^2(n+3)^3\big]. \label{sii}
\eea
Also note that
\bea
\nono&&\mbb E_{i_{[k-1]},\tau_{[k-1]}}[-g_{i_k}^k]=\sum_{s=1}^bp_s\mbb E_{\tau_{[k-1]}}[\nn\nabla_sf(z_k),\hat{x}_{k+1}^{(s)}-z_k^{(s)}\mm]\\
\nono&&\leq \min_{s\in\b}p_s\sum_{s=1}^b\mbb E_{\tau_{[k-1]}}[\nn\nabla_sf(z_k),\w{x}_{k+1}^{(s)}-z_k^{(s)}\mm]=-\min_{s\in\b}p_s[g_X^k],
\eea
and
\be
\nono\mbb E_{i_{[k-1]},\tau_{[k-1]}}[L_{i_k}D_{X_{i_k}}]=\sum_{s=1}^bp_sL_sD_{X_s}.
\ee
Taking expectation on $i_{[T]}$, $\tau_{[T]}$ on both sides of \eqref{bothsides}, and rearranging terms, we obtain
\be
\nono\min_{s\in\b}\{p_s\}\sum_{k=1}^T\aaa_k\mbb [g_X^k]\leq f(z_1)-f^{*}+(\sum_{s=1}^bp_sL_sD_{X_s})\sum_{k=1}^T\aaa_k^2+\max_{s\in\b}\{\f{p_s}{L_s}\}[\f{\w{\sigma}^2}{T_k}+\f{\mu^2}{4}L_f^2(n+4)^3].
\ee
The final result holds after noting that $P_R(k)=\aaa_k/\sum_{k=1}^T\aaa_k$, $k=1,2,...,T$.
\end{proof}
\noi We will show that after the stepsizes $\{\aaa_k\}$, smoothing parameter $\mu$, and batch sizes $T_k$ are chosen in some special way, the ZS-BCCG can achieve an $\o(b^4n/\eee^4)$ complexity bound. This result will be stated in a general way in following composite setting.

\subsection{For composite objective function}
\begin{framed}
\noi\textbf{ZS-BCCG}$'$:

\noi Replace the updating procedure in ZS-BCCG by

\noi \underline{Update $z_{k}$} by: If $s\neq i_k$, set $z_{k+1}^{(s)}=z_k^{(s)}$; else run
\bea
x_{k+1}^{(i_k)}&=& \arg\min_{y\in X_{i_k}}\big\{\nn\bar{G}_{\mu,k}^{(i_k)},y\mm +\chi_{i_k}(y)\big\}  \label{conditionc1}\\
z_{k+1}^{(i_k)}&=&(1-\aaa_k)z_k^{(i_k)}+\aaa_k x_{k+1}^{(i_k)}.\label{conditionc2}
\eea

\noi\textbf{Output} $x_R$.
\end{framed}
We define the following generalized FW-gap for solving composite problem \eqref{pro2},
\be
\nono \bar{g}_X^k=\bar{g}(\ttt(z_k))=\nn\nabla f(z_k),z_k-\ttt(z_k)\mm+\chi(z_k)-\chi(\ttt(z_k)), \ \text{where} \ \ttt(z_k)=\arg\min_{y\in X}\big\{\nn\nabla f(z_k),y\mm+\chi(y)\big\}.
\ee
For later simplicity, for $s=1,2,...,b$, we define generalized block FW-gap as follow,
\be
\nono \bar{g}_s^k=\bar{g}(\ttt_s(z_k))=\nn\nabla_s f(z_k),z_k^{(s)}-\ttt_s(z_k)\mm+\chi_s(U_s^Tz_k)-\chi_s(\ttt_s(z_k)), \ \text{where} \ \ttt_s(z_k)=\arg\min_{y\in X_s}\big\{\nn\nabla_s f(z_k),y\mm+\chi_s(y)\big\}.
\ee
\blm\label{lemg}
Let $\bar{g}_X^k$, $\bar{g}_s^k$, $s=1,2,...,b$ be defined as above, then we have
\be
\nono\bar{g}_X^k\leq\sum_{s=1}^b\bar{g}_s^k.
\ee
\elm
\begin{proof}
Observe that
\bea
\nono&&\sum_{s=1}^b\bar{g}_s^k=\nn\nabla f(z_k),z_k\mm+\chi(z_k)-\sum_{s=1}^b\big[\nn\nabla_sf(z_k),\ttt_s(z_k)\mm+\chi_s(\ttt_s(z_k))\big]\\
\nono&&\ \ \ \ \ \ \ \ \ \geq\nn\nabla f(z_k),z_k\mm+\chi(z_k)-\sum_{s=1}^b\big[\nn\nabla_sf(z_k),U_s^T\ttt(z_k)\mm+\chi_s(U_s^T\ttt(z_k))\big] \\
\nono&&\ \ \ \ \ \ \ \ \ =\bar{g}_X^k,
\eea
the proof is concluded.
\end{proof}
\begin{thm}\label{conthm2}
Under Assumptions \ref{ass1}, \ref{ass2}, \ref{ass3}, \ref{ass4}. Let $\{z_k\}_{k\geq1}$ be generated by ZS-BCCG$'$. The probability mass function $P_R$ is chosen such that $P_R(k)=Prob\{R=k\}=\aaa_k/\sum_{k=1}^T\aaa_k$. Then, for composite problem \eqref{pro2}, we have
\be
\mbb E[\bar{g}_X^R]\leq\f{\Phi(z_1)-\Phi^{*}+(\sum_{s=1}^bp_sL_sD_{X_s})\sum_{k=1}^T\aaa_k^2+\max_{s\in\b}\{p_s/L_s\}\sum_{k=1}^T\bigg[\f{\w{\sigma}^2}{T_k}+\f{\mu^2}{4}L_f^2(n+3)^3\bigg]}{(\min_sp_s)\sum_{k=1}^T\aaa_k}\label{combccg}
\ee
in which $\w{\sigma}^2$ is as in \eqref{sigma}, and $\Phi(x)=f(x)+\chi(x)$.
\end{thm}
\begin{proof}
Use the optimality condition of \eqref{conditionc1} and the convexity of $\chi_{i_k}$, it follows that, for any $u\in X_{i_k}$, there exists $h_{i_k}\in\partial\chi_{i_k}(U_{i_k}^Tx_{k+1})$ such that
\be
\nn\nabla_{i_k}f(z_k)+\bar{\Delta}_k+h_{i_k},u-x_{k+1}^{(i_k)}\mm\geq0 \ \text{and} \ \nono \chi_{i_k}(u)-\chi_{i_k}(U_{i_k}^Tx_{k+1})-\nn h_{i_k},u-x_{k+1}^{(i_k)}\mm\geq0
\ee
holds. By setting $u=\ttt_{i_k}(z_k)$ and combining above inequalities, we have
\be
\nn\nabla_{i_k}f(z_k),x_{k+1}^{(i_k)}-\ttt_{i_k}(z_k)\mm+\chi_{i_k}(U_{i_k}^Tx_{k+1})-\chi_{i_k}(\ttt_{i_k}(z_k))\leq\nn\bar{\Delta}_k,\ttt_{i_k}(z_k)-x_{k+1}^{(i_k)}\mm. \label{optine}
\ee
\bea
\nono&&f(z_{k+1})\leq f(z_k)+\nn\nabla_{i_k}f(z_k),z_{k+1}-z_k\mm+\f{L_{i_k}}{2}\|z_{k+1}^{(i_k)}-z_k^{(i_k)}\|_{i_k}^2  \\
&&=f(z_k)-\aaa_k\nn\nabla_{i_k}f(z_k),z_k^{(i_k)}-x_{k+1}^{(i_k)}\mm+\f{L_{i_k}}{2}\aaa_k^2\|z_k^{(i_k)}-x_{k+1}^{(i_k)}\|_{i_k}^2 \label{plus1}
\eea
Convexity of $\chi_{i_k}$ and \eqref{conditionc2} implies
\be
\chi_{i_k}(z_{k+1}^{(i_k)})\leq(1-\aaa_k)\chi_{i_k}(U_{i_k}^Tz_k)+\aaa_k\chi_{i_k}(U_{i_k}^Tx_{k+1}). \label{plus2}
\ee
Combining \eqref{plus1}, \eqref{plus2} and noting that $\chi(z_{k+1})-\chi(z_k)=\chi_{i_k}(U_{i_k}^Tz_{k+1})-\chi_{i_k}(U_{i_k}^Tz_{k})$, we have
\bea
\nono&&\Phi(z_{k+1})\\
\nono&&\leq\Phi(z_k)-\aaa_k\nn\nabla_{i_k}f(z_k),z_k^{(i_k)}-x_{k+1}^{(i_k)}\mm-\aaa_k\chi_{i_k}(U_{i_k}^Tz_k)+\aaa_k\chi_{i_k}(U_{i_k}^Tx_{k+1})+\f{L_{i_k}}{2}\aaa_k^2\|z_k^{(i_k)}-x_{k+1}^{(i_k)}\|_{i_k}^2 \\
\nono&&\leq\Phi(z_k)-\aaa_k\bar{g}_{i_k}^k+\aaa_k\big[\nn\nabla_{i_k}f(z_k),x_{k+1}^{(i_k)}-\ttt_{i_k}(z_k)\mm+\chi_{i_k}(U_{i_k}^Tx_{k+1})-\chi_{i_k}(\ttt_{i_k}(z_k))\big]+\f{L_{i_k}}{2}\aaa_k^2\|z_k^{(i_k)}-x_{k+1}^{(i_k)}\|_{i_k}^2\\
\nono&&\leq \Phi(z_k)-\aaa_k\bar{g}_{i_k}^k+\aaa_k\nn\bar{\Delta}_k,\ttt_{i_k}(z_k)-x_{k+1}^{(i_k)}\mm+\f{L_{i_k}}{2}\aaa_k^2\|z_k^{(i_k)}-x_{k+1}^{(i_k)}\|_{i_k}^2\\
\nono&&\leq\Phi(z_k)-\aaa_k\bar{g}_{i_k}^k+\f{L_{i_k}}{2}\aaa_k^2\big[\|z_k^{(i_k)}-x_{k+1}^{(i_k)}\|_{i_k}^2+\|\ttt_{i_k}(z_k)-x_{k+1}^{(i_k)}\|_{i_k}^2\big]\\
\nono&&\leq\Phi(z_k)-\aaa_k\bar{g}_{i_k}^k+L_{i_k}D_{X_{i_k}}^2\aaa_k^2+\f{1}{2L_{i_k}}\|\bar{\Delta}_k\|_{i_k}^2.
\eea
Note that by Lemma \ref{lemg},
\be
\mbb E_{i_{[k-1]}\tau_{[k-1]}}[\bar{g}_{i_k}^k]=\sum_{s=1}^bp_s\mbb E_{\tau_{[k-1]}}[\bar{g}_{s}^k]\geq\min_{s\in\b}p_s\mbb E_{\tau_{[k-1]}}[\bar{g}_X^k],
\ee
then, rest of the proof is similar with Theorem \ref{conthm1} and details are omitted.
\end{proof}
In  the rest of the paper, denote $\check{L}=\min_{s\in\b}L_s$. The iteration complexity of ZS-BCCG$'$ for composite optimization problem \eqref{pro2} is analyzed in following corollary.
\begin{cor}
Under assumptions of Theorem \ref{conthm2}, assume that the random variables $i_k$ are uniformly distributed. Let the smoothing parameter $\mu=\big[\f{2\check{L}\sqrt{2M^2+\sigma^2}}{5L_f^2(n+4)^3}\big]^{\f{1}{2}}$, the stepsize $\aaa_k=1/\sqrt{T}$, $k=1,2...,T$, batch sample sizes $T_k=\f{2(n+4)\sqrt{2M^2+\sigma^2}}{\check{L}}T$, $k=1,2,...,T$. Then, to find an $\epsilon$-stationary point of the composite problem \eqref{pro2}, the total number of calls to the zeroth-order oracle in ZS-BCCG$'$ is bounded by $\o(b^4n/\epsilon^4)$.
\end{cor}
\begin{proof}
Using the fact that $i_k$ are uniformly distributed with block index $b$, we have
\be
\nono\mbb E[\bar{g}_X^R]\leq\f{b[\Phi(z_1)-\Phi^{*}]}{T\aaa_1}+\sum_{s=1}^bL_sD_{X_s}\aaa_1+\f{4(n+4)[2M^2+\sigma^2]}{\check{L}T_1\aaa_1}+\f{5\mu^2L_f^2(n+4)^3}{\check{L}\aaa_1}.
\ee
Substitute the selection rule of $\{\aaa_k\}$,  $\{T_k\}$, $\mu$ into \eqref{combccg}, it follows that
\be
\nono\mbb E[\bar{g}_X^R]\leq\f{b[\Phi(z_1)-\Phi^{*}]+\sum_{s=1}^bL_sD_{X_s}+4\sqrt{2M^2+\sigma^2}}{\sqrt{T}},
\ee
which shows the desired complexity bound.
\end{proof}

Next, inspired by the technique in \cite{sliding} developed for solving convex optimization problem, we develop an approximate ZS-BCCG method for solving nonconvex optimization composite problem \eqref{pro2}. Meanwhile, we use the method to improve the above complexity bound in a performance of generalized gradient.
\begin{framed}
\noi \textbf{Approximate ZS-BCCG}:

\noindent\textbf{Input}: Initial point $x_1\in X_1\times X_2\times\cdots\times X_b$, smoothing parameter $\mu$, iteration limit $T$, stepsizes $\aaa_k$, $k\geq1$, probability mass function $P_R(\cdot)$ supported on $\{1,2,...,T\}$, probabilities $p_s\in [0,1], s=1,2,...,b$, s.t. $\sum_{s=1}^bp_s=1$,  batch sizes $T_k$ with $T_k>0$, $k\geq1$, approximating parameters $\{\delta_k\}$, $k\geq1$.

\noindent\textbf{Step} 0: Generate a random variable $i_k$ according to \eqref{distributep}
and let $R$ be a random variable with probability mass function $P_R$.

\noindent\textbf{Step} k=1,2,...,R-1: Generate $u_k=[u_{k,1},...,u_{k,T_k}]$, in which $u_{k,j}\sim N(0,I_n)$ and call the stochastic oracle to compute the $i_k$th block average stochastic gradient $\bar{G}_{\mu,k}^{(i_k)}$ by
\begin{equation}
\bar{G}_{\mu,k}^{(i_k)}=\f{1}{T_k}\sum_{t=1}^{T_k}U_{i_k}^TG_{\mu}(x_k,\xi_{k,t},u_{k,t}).\label{conG}
\end{equation}

\noi \underline{Update $x_{k}$}: If $s\neq i_k$, set $x_{k+1}^{(s)}=x_k^{(s)}$; else, run
\be
x_{k+1}^{(i_k)}=\text{CndG}_{X_{i_k}}(x_k^{(i_k)},\bar{G}_{\mu,k}^{(i_k)},\aaa_k,\delta_k),\label{cndg}
\ee
where the operator $\text{CndG}$ represents following approximate conditional gradient procedure.

\noi\textbf{Procedure} $u'=\text{CndG}_{X_s}(x,g,\aaa,\delta)$, $x\in X_s$, $g\in \mbb R^{n_s}$, $\aaa>0$, $\delta>0$, $s=1,2,...,b$:

1.Set $u_1=x\in X_s$, t=1.

2.Let
\be
v_{t+1}=\arg\min_{u\in X_s}\big\{V_{\aaa,s}(u)=\nn
g+\f{1}{\aaa}(\nabla\phi_s(u_t)-\nabla\phi_s(x)),u-u_t\mm+\chi_s(u)-\chi_s(u_t)\big\}.\label{appcnd}
\ee

3.If $V_{\aaa,s}(v_{t+1})\geq-\delta$, set $u'=u_{t}$ and terminate the procedure; else set $u_{t+1}=(1-\aaa_t)u_t+\aaa_tv_{t+1}$, where $\aaa_t=\f{2}{t+1}$.

4.Set $t\leftarrow t+1$ and go to 2.

\noi\textbf{Output $x_R$}.
\end{framed}

\begin{rmk}
Existing BCCG methods are very limited in stochastic optimization literature (\cite{b7}, \cite{b8}). To the best of our knowledge, this is the first work to implement a stochastic BCCG method by considering SZO, in contrast to the deterministic optimization in \cite{b7}, \cite{b8}. However, when the accurate information of the gradient of the objective function may be difficult to obtain, the proposed BCCG algorithms in this section are more convenient to face the challenges of stochastic circumstances. The objective functions in work \cite{b7}, \cite{b8} are all convex, that is a strong restriction in many modern practical situations like deep learning. The proposed stochastic BCCG algorithms relax the condition and also work well for nonconvex stochastic composite optimization, which is a more general setting. We mention that if the objective function in this paper is assumed to be convex, the proposed algorithms are also suitable for solving it. Moreover, it is expected to develop a stochastic cyclic BCCG method in the future for solving stochastic (nonconvex) optimization problem. Since the random block selection rule often has a better performance of convergence behavior than purely cyclic block selection rule in deterministic case (see e.g., \cite{b1}, \cite{b7}), we conjecture that this fact also holds in stochastic CG optimization setting. The strict proof is expected in the future work.
\end{rmk}

\blm
For $s=1,2,...,b$, let $P_s^{\delta}(x,g,\aaa)$ be the approximate solution of the generalized projection problem \eqref{P}
such that
\be
\nn g+\f{1}{\aaa}\big[\nabla\phi_s(P_s^{\delta}(x,g,\aaa))-\nabla\phi_s(x),u-P_s^{\delta}(x,g,\aaa)\big]\mm+\chi_s(u)-\chi_s(P_s^{\delta}(x,g,\aaa))\geq-\delta ,\forall u\in X_s. \label{app}
\ee
Then the following error estimate holds
\be
\big\|P_s(x,g,\aaa)-P_s^{\delta}(x,g,\aaa)\big\|_s^2\leq\aaa\delta, s=1,2,...,b.
\ee
\elm
\begin{proof}
In following proof, for saving space, denote $P_s=P_s(x,g,\aaa)$ and $P_s^{\delta}=P_s^{\delta}(x,g,\aaa)$. The optimality condition implies that there exists $h_s\in\partial\chi_s(P_s)$ such that
\be
\nn g+\f{1}{\aaa}\big[\nabla\phi_s(P_s)-\nabla\phi_s(x)\big]+h_s,u_1-P_s\mm\geq0, \ u_1\in X_s. \label{u_1}
\ee
By setting $u_1=P_s^{\delta}$ in \eqref{u_1} and $u=P_s$ in \eqref{app}, and adding them together, we have
\be
\nn\f{1}{\aaa}\big[\nabla\phi_s(P_s)-\phi_s(P_s^{\delta})\big]+h_s,P_s^{\delta}-P_s\mm+\chi_s(P_s)-\chi_s(P_s^{\delta})\geq-\delta. \label{add}
\ee
Note that by convexity of function $\chi_s$ at $P_s$, we have
\be
\chi_s(P_s^{\delta})-\chi_s(P_s)\geq \nn h_s,P_s^{\delta}-P_s\mm. \label{chicon}
\ee
Combine \eqref{add} with \eqref{chicon} and rearrange terms, we have
\be
\nn\nabla\phi_s(P_s)-\nabla\phi_s(P_s^{\delta}),P_s-P_s^{\delta}\mm\leq\aaa\delta.
\ee
The desired result is obtained after using the 1-strong convexity of $\phi_s$.
\end{proof}
Now it's ready to present the main result of the approximate ZS-BCCG method for nonconvex composite optimization. The result is also based on an i.i.d randomization selection scheme on block coordinate. The following theorem provides the estimate on generalized gradient in approximate ZS-BCCG method in terms of approximating parameters $\{\delta_k\}$, stepsizes $\{\aaa_k\}$, dimension $n$, Lipschitz constants $L_s$, $s=1,2,...,b$, $L_f$, and block coordinate probabilities $p_s$, $s=1,2,...,b$. The theorem is the foundation to further achieve rate of convergence and complexity bound.
\begin{thm}\label{thmacnd}
Under Assumptions \ref{ass1}, \ref{ass2}, \ref{ass3}. Suppose the stepsizes $\{\aaa_k\}$, $k\geq1$ in approximate ZS-BCCG satisfy $\aaa_k\leq1/L_s$, $s\in \b$. The probability mass function $P_R$ is chosen such that
\be
P_R(k):=Prob\big\{R=k\big\}=\f{\aaa_k\min_{s\in\b}\{ p_s(1-L_s\aaa_k)\}}{\sum_{k=1}^T\aaa_k\min_{s\in\b} \{p_s(1-L_s\aaa_k)\}}, \ k=1,2,...,T.\label{masscond}
\ee
Then we have
\be
\mbb E\big[\|\p(x_R,\nabla f(x_R),\aaa_R)\|^2\big]\leq\f{2[\Phi(x_1)-\Phi^{*}]+6\sum_{k=1}^T\delta_k+\sum_{k=1}^T\max_{s\in\b}\big\{p_s(\f{1}{L_s}+4\aaa_k)\big\}\big[\f{2\w{\sigma}^2}{T_k}+\f{\mu^2}{2}L_f^2(n+3)^3\big]}{\sum_{k=1}^T\aaa_k\min_{s\in\b}\big\{p_s(1-L_s\aaa_k)\big\}} \label{maincnd}
\ee
\end{thm}
\begin{proof}
Use the block Lipschitz property of $f$, it follows that
\be
f(x_{k+1})\leq f(x_k)+\nn\nabla_{i_k}f(x_k),x_{k+1}^{(i_k)}-x_k^{(i_k)}\mm+\f{L_{i_k}}{2}\|x_{k+1}-x_k\|_{i_k}^2.\label{com1}
\ee
Also note that, by \eqref{cndg}, we have
\be
\nn\bar{G}_{\mu,k}^{(i_k)}+\f{1}{\aaa_k}\big[\nabla\phi_{i_k}(x_{k+1}^{(i_k)})-\nabla\phi_{i_k}(x_{k}^{(i_k)})\big],u-x_{k+1}^{(i_k)}\mm+\chi_{i_k}(u)-\chi_{i_k}(x_{k+1}^{(i_k)})\geq-\delta_k, \ \forall x\in X_{i_k}. \label{com2}
\ee
By setting $u=x_k^{(i_k)}$ in \eqref{com2}, summing it up with \eqref{com1}, using the 1-strong convexity of $\phi_{i_k}$ and noting that $\bar{\Delta}_k=\bar{G}_{\mu,k}^{(i_k)}(x_k)-\nabla_{i_k}f(x_k)$, we have
\bea
\nono f(x_{k+1})\leq f(x_k)+\chi_{i_k}(x_k^{(i_k)})-\chi_{i_k}(x_{k+1}^{(i_k)})-\big(\f{1}{\aaa_k}-\f{L_{i_k}}{2}\big)\|x_{k+1}^{(i_k)}-x_{k}^{(i_k)}\|_{i_k}^2+\nn\bar{\Delta}_k,x_k^{(i_k)}-x_{k+1}^{(i_k)}\mm+\delta_k,
\eea
which together with the fact $\chi_{i_k}(x_k^{(i_k)})-\chi_{i_k}(x_{k+1}^{(i_k)})=\chi(x_k)-\chi(x_{k+1})$, $\Phi(x)=f(x)+\chi(x)$ and $\nn\bar{\Delta}_k,x_k^{(i_k)}-x_{k+1}^{(i_k)}\mm\leq \f{L_{i_k}}{2}\|x_{k+1}^{(i_k)}-x_k^{(i_k)}\|_{i_k}^2+\f{\|\bar{\Delta}_k\|_{i_k}^2}{2L_{i_k}}$ implies
\be
\nono\Phi(x_{k+1})\leq\Phi(x_k)-\big(\f{1}{\aaa_k}-L_{i_k}\big)\|x_{k+1}^{(i_k)}-x_{k}^{(i_k)}\|_{i_k}^2+\f{\|\bar{\Delta}_k\|_{i_k}^2}{2L_{i_k}}+\delta_k.
\ee
Sum up both sides of the above inequality from $k=1$ to $T$, rearrange terms and note that $\Phi(x_{T+1})\geq\Phi^{*}$, we have
\be
\sum_{k=1}^T\big(\f{1}{\aaa_k}-L_{i_k}\big)\|x_{k+1}^{(i_k)}-x_k^{(i_k)}\|_{i_k}^2\leq\Phi(x_1)-\Phi^{*}+\f{1}{2}\sum_{k=1}^T\f{\|\bar{\Delta}_k\|_{i_k}^2}{L_{i_k}}+\sum_{k=1}^T\delta_k.\label{guoduqian}
\ee
On the other hand, note the fact that $x_{k+1}^{(i_k)}=P_{i_k}^{\delta_k}(x_k^{(i_k)},\bar{G}_{\mu,k}^{(i_k)},\aaa_k)$,
\bea
\nono&&\aaa_k^2\big\|\p_{i_k}(x_k,\nabla f(x_k),\aaa_k)\big\|_{i_k}^2=\big\|x_k^{(i_k)}-P_{i_k}(x_k^{(i_k)},\nabla_{i_k}f(x_k),\aaa_k)\big\|_{i_k}^2 \\
\nono&&=\big\|x_k^{(i_k)}-x_{k+1}^{(i_k)}+P_{i_k}^{\delta_k}(x_k^{(i_k)},\bar{G}_{\mu,k}^{(i_k)},\aaa_k)-P_{i_k}(x_k^{(i_k)},\bar{G}_{\mu,k}^{(i_k)},\aaa_k) \\
\nono&&\ \ \ \ +P_{i_k}(x_k^{(i_k)},\bar{G}_{\mu,k}^{(i_k)},\aaa_k)-P_{i_k}(x_k^{(i_k)},\nabla_{i_k}f(x_k),\aaa_k)\big\|_{i_k}^2 \\
&&\leq 2\|x_k^{(i_k)}-x_{k+1}^{(i_k)}\|_{i_k}^2+4\aaa_k\delta_k+4\aaa_k^2\|\bar{\Delta}_k\|_{i_k}^2. \label{guodu}
\eea
Multiplying both sides of \eqref{guodu} by $\f{\aaa_k-L_{i_k}\aaa_k^2}{2\aaa_k^2}$ and combining it with \eqref{guoduqian}, we have
\bea
\nono&&\sum_{k=1}^T\big(\f{\aaa_k-L_{i_k}\aaa_k^2}{2}\big)\big\|\pf\big\|_{i_k}^2  \\
&&\leq\Phi(x_1)-\Phi^{*}+\sum_{k=1}^T\big(3-2L_{i_k}\aaa_k\big)\delta_k+\sum_{k=1}^T\big(\f{1}{2L_{i_k}}+2\aaa_k-2L_{i_k}\aaa_k^2\big)\|\bar{\Delta}_k\|_{i_k}^2,
\eea
which in view of the fact that $L_{i_k}\aaa_k\geq0$ and $L_{i_k}\aaa_k^2\geq0$ implies that
\be
\sum_{k=1}^T\big(\aaa_k-L_{i_k}\aaa_k^2\big)\big\|\pf\big\|_{i_k}^2\leq2\big[\Phi(x_1)-\Phi^{*}\big]+6\sum_{k=1}^T\delta_k+\sum_{k=1}^T\big(\f{1}{L_{i_k}}+4\aaa_k\big)\|\bar{\Delta}_k\|_{i_k}^2.\label{expqian}
\ee
Note that
\bea
\nono&&\mbb E_{i_{[k-1]},\tau_{[k-1]}}\bigg[(\aaa_k-L_{i_k}\aaa_k^2)\|\pf\|_{i_k}^2\bigg]  \\
\nono&&=\aaa_k\sum_{s=1}^bp_s(1-L_s\aaa_k)\mbb E_{\tau_{[k-1]}}\big[\|\p_s(x_k,\nabla f(x_k),\aaa_k)\|_s^2\big] \\
\nono&&\geq\aaa_k\min_{s\in\b}p_s(1-L_s\aaa_k)\bigg(\sum_{s=1}^b\mbb E_{\tau_{[k-1]}}\big[\|\p_s(x_k,\nabla f(x_k),\aaa_k)\|_s^2\big]\bigg) \\
\nono&&\geq\aaa_k\min_{s\in\b}p_s(1-L_s\aaa_k)\mbb E_{\tau_{[k-1]}}\big[\|\p(x_k,\nabla f(x_k),\aaa_k)\|^2\big],
\eea
and
\bea
\nono\mbb E_{i_{[k-1]},\tau_{[k-1]}}\bigg[\big(\f{1}{L_{i_k}}+4\aaa_k\big)\|\bar{\Delta}_k\|_{i_k}^2\bigg]\leq\max_{s\in\b}\big\{p_s\big(\f{1}{L_s}+4\aaa_k\big)\big\}\big[\f{2\w{\sigma}^2}{T_k}+\f{\mu^2}{2}L_f^2(n+3)^3\big],
\eea
which follows from the similar reason with \eqref{sii}. Take total expectation over $i_{[T]},\tau_{[T]}$ on both sides of \eqref{expqian} and we conclude the proof.
\end{proof}
\begin{rmk}
Theorem \ref{thmacnd} provides an important convergence result of approximate ZS-BCCG in a generalized gradient performance. It can be treated  as a counterpart of Theorem \ref{general main} for ZS-BMD. However, there are several quite different restrictions on parameters. The available range for the stepsizes in two methods are $(0,2/L_s]$ and $(0,1/L_s]$ respectively. In contrast to ZS-BMD, the approximate parameters $\delta_k$ in approximating ZS-BCCG highlight the importance of approximation technique used in ZS-BCCG which ZS-BMD in last section does not implement directly.
\end{rmk}
The selection details of the related parameters of approximate ZS-BCCG is analyzed in the following corollary.

\begin{cor}\label{cor6}
Under assumptions of Theorem \ref{thmacnd}, denote $D_{\Phi}=[(\Phi(x_1)-\Phi^{*})/\hat{L}]^{\f{1}{2}}$. Suppose that in the approximate ZS-BCCG, stepsizes $\{\aaa_k\}$ are taken as $\aaa_k=1/2\hat{L}$, approximating parameters $\{\delta_k\}$  are taken as $\delta_k=\delta=1/3T$, random varibles $\{i_k\}$ are uniformly distributed with block index $b$ ($p_1=p_2=\cdots=p_b=1/b$), batch sample size of each step is taken as $T_k=T'$, and the block stochastic gradient $\bar{G}_{\mu,k}^{(i_k)}$ is computed as in \eqref{conG}, then we have
\be
\mbb E\big[\|\p(x_R,\nabla f(x_R),\aaa_R)\|^2\big]\leq\f{4b\hat{L}^2D_{\Phi}^2+4b\hat{L}}{T}+\big(\f{4\hat{L}}{\check{L}}+8\big)\f{\w{\sigma}^2}{T'}+\big(\f{\hat{L}}{\check{L}}+2\big)\mu^2L_f^2(n+3)^3. \label{cor66}
\ee
\end{cor}
\begin{proof}
The result follows directly by noting that
\be
\nono\sum_{k=1}^T\aaa_k\min_{s\in\b}\big\{p_s(1-L_s\aaa_k)\big\}\geq\sum_{k=1}^T\f{1}{2\hat{L}}\cdot\f{1}{b}=\f{T}{2b\hat{L}},
\ee
\be
\nono\max_{s\in\b}\big\{p_s(\f{1}{L_s}+4\aaa_k)\big\}\leq\f{1}{b\check{L}}+\f{2}{b\hat{L}},
\ee
and substituting them into \eqref{maincnd}, and rearranging terms.
\end{proof}

\begin{rmk}
Note that in the above ZS-BCCG, when $\mu$ is taken as $\mu=O(1/(n+4)^{\f{3}{2}}\sqrt{T})$, batch sample size $T'$ is taken as $T'=O((n+4)T)$, an $\o(b^2n/\eee^2)$ complexity bound is obtained for finding an $\eee$-stationary point of the composite problem \eqref{pro2}. This complexity bound coincides with the complexity bound in Remark \ref{comrmk1} of Corollary \ref{cor2} for ZS-BMD in last section. Although they theoretically share the same rate, however, as a result of the appearance of the Bregman divergence in the algorithm structure of ZS-BMD, solving subproblem \eqref{BMD} can be generally much more difficult than subproblem \eqref{cndg} due to the expensive computation. Moreover, the above fact become more obvious when the block regularization function $\chi_s(x)$ is taken as some examples with special convex structure like entropy function $\chi_{s}(x^{(s)})=\sum_{s=s_1}^{s=s_k}x^{s}\log x^{s}$, $(x^{s_1},...,x^{s_k})$ denote the component of block variable $x^{(s)}$; $\chi_s(x^{(s)})=l\|x^{(s)}\|_1 \ \text{or} \ l\|x^{(s)}\|_{\infty}$ for some $l>0$; or $p$-norm $\chi_s(x^{(s)})=\f{1}{2}\|x^{(s)}\|_p^2$, $p\in (1,2]$. In these situations, it would be better to consider the proposed approximate ZS-BCCG method.
\end{rmk}
We establish the convergence result for approximate ZS-BCCG which has an explicit representation in terms of the total calls of the zeroth-order stochastic oracle $\w{T}$ after selection of $\mu$ and $T'$ in terms of $\w{T}$. The result is the foundation to further establish the two-phase ZS-BCCG method. In the rest of the paper, we denote $\ooo_L=\hat{L}/\check{L}+2$.
\begin{cor}
Under assumptions of Corollary \ref{cor6}, suppose the smoothing parameter $\mu$ satisfies
\be
\mu\leq \f{D_{\Phi}}{n+4}\sqrt{\f{b}{\w{T}}},\label{condmu}
\ee
the number of calls to SZO at each iteration of ZS-BCCG is
\be
T'=\bigg\lceil\min\bigg\{\max\bigg\{\f{\ooo_L\sqrt{(n+4)(2M^2+\sigma^2)\w{T}}}{\w{D}},\ooo_L(n+4)\bigg\},\w{T}\bigg\}\bigg\rceil,  \label{batchsizec}
\ee
for some $\w{D}>0$, then the generalized gradient in ZS-BCCG satisfies $\f{1}{\ooo_{L}b}\mbb E\big\{\|\p(x_R,\nabla f(x_R),\aaa_R)\|^2\big\}\leq\textbf{C}_{\w{T}}$, where
\be
\textbf{C}_{\w{T}}=\f{8\sqrt{(n+4)(2M^2+\sigma^2)}}{\sqrt{\w{T}}}\bigg(2\kkk_1\w{D}+\f{\hat{L}^2D_{\Phi}^2+\hat{L}}{\w{D}}\bigg)+\f{\big[(16\kkk_2+1)L_f^2D_{\Phi}^2+8(\hat{L}^2D_{\Phi}^2+\hat{L})\big](n+4)}{\w{T}},\label{ct}
\ee
and
\be
\kkk_1=\max\bigg\{\f{\ooo_L\sqrt{(n+4)(2M^2+\sigma^2)}}{\w{D}\sqrt{\w{T}}},1\bigg\}, \  \kkk_2=\max\bigg\{\f{\ooo_L(n+4)}{\w{T}},1\bigg\}.
\ee
\end{cor}
\begin{proof}
Since $b\geq1$, $T\geq \w{T}/2T'$, substitute $\w{\sigma}^2=4(n+4)(2M^2+\sigma^2+\mu^2L_f^2(n+3)^3)$ into \eqref{cor66}, we have
\bea
\nono&&\eqref{cor66}=\f{4b\hat{L}^2D_{\Phi}^2+4b\hat{L}}{T}+\f{4\ooo_{L}b\w{\sigma}^2}{T'}+\ooo_L\mu^2L_f^2(n+3)^3 \\
\nono&&\ \ \ \ \ \ \ \  \leq\f{(8b\hat{L}^2D_{\Phi}^2+8b\hat{L})T'}{\w{T}}+\f{16\ooo_L(n+4)(2M^2+\sigma^2)b}{T'} \\
\nono&&\ \ \ \ \ \ \ \ \ \ \ +\f{16\ooo_L(n+4)L_f^2D_{\Phi}^2b}{T'\w{T}}+\f{\ooo_LL_f^2D_{\Phi}^2(n+4)b}{\w{T}} \\
\nono&&\ \ \ \ \ \ \ \ \leq\f{8b\hat{L}^2D_{\Phi}^2+8b\hat{L}}{\w{T}}\bigg[\f{\ooo_L\sqrt{(n+4)(2M^2+\sigma^2)\w{T}}}{\w{D}\kkk_1}+\f{\ooo_L(n+4)}{\kkk_2}\bigg] \\
\nono&&\ \ \ \ \ \ \ \ \ \ \ + \f{16\kkk_1\ooo_L\w{D}b\sqrt{(n+4)(2M^2+\sigma^2)}}{\sqrt{\w{T}}}+\f{16\kkk_2\ooo_LL_f^2D_{\Phi}^2(n+4)b}{\w{T}}+\f{\ooo_LL_f^2D_{\Phi}^2(n+4)b}{\w{T}} \\
\nono&&\ \ \ \ \ \ \ \ \leq\f{(8b\hat{L}^2D_{\Phi}^2+8b\hat{L})\ooo_L\sqrt{(n+4)(2M^2+\sigma^2)})}{\w{D}\sqrt{\w{T}}}+\f{16\kkk_1\ooo_L\w{D}b\sqrt{(n+4)(2M^2+\sigma^2)}}{\sqrt{\w{T}}} \\
\nono&&\ \ \ \ \ \ \ \ \ \ \ +\f{(8b\hat{L}^2D_{\Phi}^2+8b\hat{L})\ooo_L(n+4)}{\w{T}}+\f{16\kkk_2\ooo_LL_f^2D_{\Phi}(n+4)b}{\w{T}}+\f{\ooo_LL_f^2D_{\Phi}^2(n+4)b}{\w{T}}.
\eea
Then desired result follows after rearranging terms and dividing both sides by $\ooo_Lb$.
\end{proof}

Now we propose to establish the complexity results of approximate ZS-BCCG for finding an $(\eee,\Lambda)$-solution of problem \eqref{pro2}. Markov inequality directly implies
\be
\nono Prob\big\{\|\p(x_R,\nabla f(x_R),\aaa_R)\|^2\geq\lambda b\ooo_L\textbf{C}_{\w{T}}\big\}\leq\f{1}{\lambda},\ \forall\lambda\geq0.
\ee
For any $\eee>0$ and $\Lambda\in(0,1)$, by setting $\lambda=\f{1}{\Lambda}$ and
\bea
\nono&&\w{T}=\bigg\lceil\max\bigg\{\ooo_L(n+4),\f{\ooo_L^2(n+4)(2M^2+\sigma^2)}{\w{D}^2}, \f{2\ooo_Lb[17L_f^2D_{\Phi}^2+8(\hat{L}D_{\Phi}^2+\hat{L})](n+4)}{\Lambda\eee}, \\
\nono&&\ \ \ \ \ \ \ \ \ \ \ \ \ \ \f{16^2\ooo_L^2b^2(n+4)(2M^2+\sigma^2)}{\Lambda^2\eee^2}\bigg(2\w{D}^2+\f{\hat{L^2}D_{\Phi}^2+\hat{L}}{\w{D}}\bigg)^2\bigg\}\bigg\rceil
\eea
in the above inequality, we obtain that, the complexity of $\w{T}$ for finding an $(\eee,\Lambda)$-solution in approximate ZS-BCCG method, after disregarding several constant factors, can be bounded by
\be
\o\bigg\{\f{\ooo_Lbn(L_f^2D_{\Phi}^2+\hat{L}D_{\Phi}^2+\hat{L})}{\Lambda\eee}+\f{\ooo_L^2b^2n(M^2+\sigma^2)}{\Lambda^2\eee^2}\bigg(\w{D}+\f{\hat{L}^2D_{\Phi}^2+\hat{L}}{\w{D}}\bigg)^2\bigg\}. \label{comple2}
\ee

\section{Two-phase ZS-BCCG optimization scheme}
To improve the complexity \eqref{comple2} obtained by using approximate ZS-BCCG in single run, we design a two-phase ZS-BCCG optimization procedure by following the approximate ZS-BCCG method.

\begin{framed}
\noindent\textbf{Two-phase ZS-BCCG:}

\noi\textbf{Input}: Initial point $x_1\in X_1\times X_2\times\cdots\times X_b$, generate the random variables $\{i_k\}$ that are uniformly distributed ($p_1=p_2=\cdots=p_b=1/b$), number of runs $S$, total $\w{T}$ of calls to SZO in each run of approximate ZS-BCCG, sample size $\t$ in the post-optimization phase.

\noi\textbf{Optimization phase}:

\noindent For $i=1,2,...,S$, call the approximate ZS-BCCG algorithm with input $x_1$, batch sizes $T_k=T'$ which is as in \eqref{batchsizec}, iteration limit $T=\lfloor\w{T}/T'\rfloor$, stepsizes $\aaa_k=1/2\hat{L}$, $k=1,2,...,T$, smoothing parameter $\mu$ satisfying \eqref{condmu} and approximating parameters $\delta_k=1/3T$, probability function $P_R$ in \eqref{masscond}. Output $\bar{x}_i=x_{R_i}$, $i=1,2,...,S$.

\noi\textbf{Post-optimization phase}:

\noindent Select a solution $\bar{x}_{*}$ from the candidate list $\{\bar{x}_1,\bar{x}_2,...,\bar{x}_S\}$ such that
\be
\bar{x}_{*}=\arg\min_{i=1,2,...,S}\big\|\bar{\g}_{G}(\bar{x}_{i})\big\|, \ \bar{\g}_{G}(\bar{x}_{i})=\p(\bar{x}_i,G_{\mu,\t}(\bar{x}_i),\aaa_{R_i}),
\ee
in which $G_{\mu,\t}(x)=\f{1}{\t}\sum_{k=1}^{\t}G_{\mu}(x,\xi_k,u_k)$.

\noi\textbf{Output $\bar{x}_{*}$.}
\end{framed}

\begin{thm}\label{promain2}
Under Assumptions \ref{ass1}, \ref{ass2}, \ref{ass3}. Let $\textbf{C}_{\w{T}}$ be defined as in \eqref{ct}, $b$ is the total block number. Then the 2-ZS-BCCG method for constrained composite optimization problem \eqref{pro2} satisfies
\bea
\nono&&Prob\bigg\{\big\|\p(\bar{x}_{*},\nabla f(\bar{x}_{*}),\aaa_{R_{*}})\|^2\geq16b\ooo_L\textbf{C}_{\w{T}}+\f{3D_{\Phi}^2L_f^2b(n+4)}{\w{T}}\\
&&\ \ \ \ \ \ \ \ \ \ \ +\f{32(n+4)\lambda}{\t}\bigg[2M^2+\sigma^2+\f{D_{\Phi}^2L_f^2b}{\w{T}}\bigg]\bigg\}\leq\f{S+1}{\lambda}+2^{-S}, \ \forall\lambda>0.
\eea
\end{thm}
\begin{proof}
The proof follows from the similar procedure with Theorem \ref{promain1} once the parameters $\ooo_L$, $\textbf{C}_{\w{T}}$ are replaced in appropriate positions. For saving space, the detailed proof is omitted.
\end{proof}
In what follows, we give a detailed selection of parameters for 2-ZS-BCCG method. The following parameters selection rule achieves an improved complexity than \eqref{comple2} on confidence level $\Lambda$.
\begin{cor}\label{compcom}
Under assumptions of Theorem \ref{promain2}, let $\eee>0$ and $\Lambda\in(0,1)$, suppose the parameters $(S,\w{T},\t)$ are selected as
\bea
\nono&&S_c(\Lambda)=\big\lceil \log_2(\f{2}{\Lambda})\big\rceil,\label{Sc}\\
\nono&&\w{T}_c(\epsilon)=\bigg\lceil\max\bigg\{\ooo_L(n+4), \f{\ooo_L^2(n+4)(2M^2+\sigma^2)}{\w{D}^2},\f{8^2\ooo_Lb(n+4)[17L_f^2D_{\Phi}^2+8(\hat{L}^2D_{\Phi}^2+\hat{L})]}{\epsilon},\\
\nono&&\ \ \ \ \ \ \ \ \ \  b^2\ooo_L^2\bigg [\f{8\cdot32\sqrt{(n+4)(2M^2+\sigma^2)}}{\epsilon}\bigg(2\w{D}+\f{\hat{L}^2D_{\Phi}^2+\hat{L}}{\w{D}}\bigg)\bigg]^2\bigg\} \bigg\rceil, \label{wtc}\\
\nono&&\t_c(\epsilon,\Lambda)=\bigg\lceil\f{32(n+4)\cdot2(S+1)}{\Lambda}\max\bigg\{1,\f{16(2M^2+\sigma^2)}{\epsilon}\bigg\} \bigg\rceil,  \label{gtc}
\eea
then the 2-ZS-BCCG computers an $(\eee,\Lambda)$-solution of problem \eqref{pro2} after taking at most total number of calls
\be
S_c(\Lambda)[\w{T}_c(\eee)+\t_c(\eee,\Lambda)] \label{comp1}
\ee
to SZO.
\end{cor}
\begin{proof}
By using the selection rule of $S_c(\Lambda)$, $\w{T}_c(\eee)$, $\t_c(\eee,\Lambda)$, we have
\bea
\nono&&\textbf{C}_{\w{T}}\leq\f{8\sqrt{(n+4)(2M^2+\sigma^2)}}{\sqrt{\w{T}}}\bigg(2\w{D}+\f{\hat{L}^2D_{\Phi}^2+\hat{L}}{\w{D}}\bigg)+\f{(n+4)\big[17L_f^2D_{\Phi}^2+8(\hat{L}^2D_{\Phi}^2+\hat{L})\big]}{\w{T}}\\
\nono&&\ \ \ \ \ \leq\f{\eee}{32\ooo_Lb}+\f{\eee}{64\ooo_Lb}=\f{3\eee}{64\ooo_Lb}=\f{3\eee}{64},
\eea
which also implies that
\be
\nono 16b\ooo_L\textbf{C}_{\w{T}}+\f{3D_{\Phi}^2L_f^2b(n+4)}{\w{T}}\leq\f{3\eee}{4}+\f{\eee}{8^2\cdot17}<\f{7\eee}{8}.
\ee
By setting $\lambda=2(S+1)/\Lambda$, using the selection rule of $S_c(\Lambda)$, $\w{T}_c(\eee)$, $\t_c(\eee,\Lambda)$ again, and noting that $\ooo_L\geq3$, we have
\be
\nono \f{32(n+4)\lambda}{\t}\bigg[\f{D_{\Phi}^2L_f^2b}{\w{T}}+2M^2+\sigma^2\bigg]\leq\f{\eee}{17\times8^2}+\f{\eee}{16}<\f{\eee}{8}.
\ee
Finally, it follows that
\bea
\nono &&Prob\big\{\|\p(\bar{x}_{*},\nabla f(\bar{x}_{*}),\aaa_{R_{*}})\|^2\geq\eee\big\}\leq Prob\bigg\{\|\p(\bar{x}_{*},\nabla f(\bar{x}_{*}),\aaa_{R_{*}})\|^2\geq16b\ooo_L\textbf{C}_{\w{T}_c(\eee)} \\
\nono&&\ \ \ \ \ \ \ \ \ \ \ \ \ \ \ \ \ \ \ \ \ \ \ \ \ \ \ \ \ \ \ \ \ +\f{3D_{\Phi}^2L_f^2b(n+4)}{\w{T}_c(\eee)} +\f{32(n+4)\lambda}{\t_c(\eee,\Lambda)}\bigg[2M^2+\sigma^2+\f{D_{\Phi}^2L_f^2b}{\w{T}_c(\eee)}\bigg] \bigg\} \\
\nono&&\ \ \ \ \ \ \ \ \ \ \ \ \ \ \ \ \ \ \ \ \ \ \ \ \ \ \ \ \ \ \ \ \ \ \ \ \ \ \ \ \ \ \ \ \ \ \leq \f{S+1}{\lambda}+2^{-S}\leq \f{\Lambda}{2}+\f{\Lambda}{2}=\Lambda,
\eea
which concludes the proof.
\end{proof}
Corollary \ref{compcom} indicates that the complexity bound in \eqref{comp1} of 2-ZS-BCCG for finding an $(\eee,\Lambda)$-solution of composite problem \eqref{pro2} for generalized gradient, can be bounded by
\bea
\nono&&\o\bigg\{\f{bn\ooo_L[(L_f^2+\hat{L}^2)D_{\Phi}^2+\hat{L}]}{\eee}\log_2\bigg(\f{1}{\Lambda}\bigg)+\f{(M^2+\sigma^2)n}{\Lambda\eee}\log_2^2\bigg(\f{1}{\Lambda}\bigg) \\
\nono&&\ \ \ \ \ \ \ \ \ \ +\f{b^2n\ooo_L^2(M^2+\sigma^2)}{\eee^2}\bigg(\w{D}+\f{\hat{L}^2D_{\Phi}^2+\hat{L}}{\w{D}}\bigg)\log_2\bigg(\f{1}{\Lambda}\bigg)\bigg\},
\eea
which stands as one of the main contributions of this work. The appearance of $\ooo_L$ shows that, in contrast to ZS-BCD and ZS-BMD, the the complexity of ZS-BCCG depends not only on the upper bound $\hat{L}$ of block Lipschitz estimations of objective function $f$ but also their lower bound $\check{L}$. If we do not take the effect of estimations of $L_s$, $s=1,2,...,b$, $L_f$ and $D_{\Phi}$ into consideration for the moment, the above complexity bound results in \eqref{commain}, improving the complexity bound \eqref{comb1} of ZS-BCD for unconstrained optimization. To the best of our knowledge, this is the first work to consider the two-phase BCCG technique to achieve $(\eee,\Lambda)$-solution in stochastic optimization literature. The convergence and complexity results are new. Meanwhile, the results are convenient for nonconvex setting, where BCCG methods have not been proposed before.

\section{Concluding remarks}
In this work, we develop several new classes of zeroth-order block coordinate type algorithms for solving nonconvex optimization problems and analyze them in several aspects. Specifically, by incorporating randomization scheme, we first develop ZS-BCD algorithm for solving classical unconstrained nonconvex stochastic optimization problem, then ZS-BMD algorithm and ZS-BCCG algorithm for solving constrained nonconvex stochastic composite optimization problem. For each of the algorithms, the rate of convergence and corresponding complexity bound for finding $\eee$-stationary point are achieved. To improve the complexity results which directly comes from corresponding Markov inequality for finding $(\eee,\Lambda)$-solution, we further develop two-phase optimization schemes for both of these two classes of methods. The improved explicit complexity bounds are achieved.  These complexity results are new in block coordinate method literature. The methods are suitable for optimization problems when only stochastic zeroth-order information is available. The analysis in this work has shown considerable theoretical value of the proposed methods and the potential practical value is expected to be further explored in the future.

Some future work based on this work may be considered. (1) In this paper, as a whole, selection rules of random variables $i_k$ of all algorithms are in an i.i.d manner. Note that, in some application setting like distributed optimization, sometimes i.i.d random and cyclic selections are infeasible or quite costly. It is necessary to consider a non-i.i.d and non-cyclic selection rule for $i_k$ to overcome the difficulty. A typical variable example is Markov chain $i_k$ (see e.g., \cite{b5}). It is interesting to consider the possibility to propose a zeroth-order stochastic Markov chain block coordinate type method. It is also technically not easy to make this extension. (2) Recently, the class of weakly smooth functions has become popular to act as a basic function class to present theoretical optimization results. It is expected that the block two-phase results in this paper can be further extended to optimization problem of nonconvex weakly smooth functions. (3) It is possible to apply ZS-BMD and ZS-BCCG to obtain convergence results for convex case. Accelerated ZS-BMD and ZS-BCCG algorithms are possible to be established for convex case. (4) Online optimization has become one of the core topics in machine learning research. It is expected to develop the proposed methods to online setting to face increasing challenges from statistical machine learning.

\section{Appendix}

\noi \textbf{Proof of Lemma \ref{lemma1}:}
\begin{proof}
By using the optimality condition of \eqref{P}, there exists an $h\in\partial\chi_s(U_s^Tx^{+})$, such that
\be
\nono\nn g^{(s)}+\f{1}{\aaa}\big[\nabla\phi_s(U_s^Tx^{+})-\nabla\phi_s(U_s^Tx)\big]+h,\f{1}{\aaa}\big(y-U_s^Tx^{+}\big)\mm\geq0, \ \forall y\in X_s.
\ee
Set $y=U_s^Tx$ in above inequality, it follows that
\bea
\nono&&\nn g^{(s)}, \p_s(x,U_sg^{(s)},\aaa)\mm\geq\f{1}{\aaa^2}\nn\nabla\phi_s(U_s^Tx^+)-\nabla\phi_s(U_s^Tx), U_s^T(x^+-x)\mm \\
&&\ \ \ \ \ \ \ \ \ \ \ \ \ \ \ \ \ \ \ \ \ \ \ \ \ \ \ \ \ \ \ \ +\f{1}{\aaa}\big[\chi_s(U^Tx^+)-\chi_s(U_s^Tx)\big].
\eea
Then the lemma follows by using the 1-strong convexity of $\phi_s$ and the definition of $\p_s(x,g,\aaa)$.
\end{proof}
For saving space, in following two proofs, denote $x_1^+=P(x,g_1,\aaa)$ and $x_2^+=P(x,g_2,\aaa)$.

\noi \textbf{Proof of Lemma \ref{lemmap}:}
\begin{proof}
By using the optimality condtion of \eqref{P}, there exist $h_1\in\partial\chi_s(U_s^Tx_1^+)$, and $h_2\in\partial\chi_s(U_s^Tx_2^+)$ such that
\bea
&&\nn g_1^{(s)}+\f{1}{\aaa}\big[\nabla\phi_s(U_s^Tx_1^+)-\nabla\phi_s(U_s^Tx)\big]+h_1,y-U_s^Tx_1^+\mm\geq0, \ \forall y \in X_s, \label{opt1}\\
&&\nn g_2^{(s)}+\f{1}{\aaa}\big[\nabla\phi_s(U_s^Tx_2^+)-\nabla\phi_s(U_s^Tx)\big]+h_2,y-U_s^Tx_2^+\mm\geq0, \ \forall y \in X_s.  \label{opt2}
\eea
Let $y=U_s^Tx_2^+$ in \eqref{opt1}, the convexity of $\chi_s$ implies
\bea
\nono&&\nn g_1^{(s)},U_s^T(x_2^+-x_1^+)\mm\\
\nono&&\geq\f{1}{\aaa}\nn\nabla\phi_s(U_s^Tx)-\nabla\phi_s(U_s^Tx_1^+),U_s^Tx_2^+-U_s^Tx_1^+\mm+\nn h_1,U_s^Tx_2^+-U_s^Tx_1^+\mm \\
\nono&&\geq\f{1}{\aaa}\nn\nabla\phi_s(U_s^Tx)-\nabla\phi_s(U_s^Tx_2^+),U_s^Tx_2^+-U_s^Tx_1^+\mm+\chi_s(U_s^Tx_1^+)-\chi_s(U_s^Tx_2^+)\\
&&\ \ \ +\f{1}{\aaa}\nn\nabla\phi_s(U_s^Tx_2^+)-\nabla\phi_s(U_s^Tx_1^+),U_s^Tx_2^+-U_s^Tx_1^+\mm. \label{opt01}
\eea
Similarly, by setting $y=U_s^Tx_1^+$ in \eqref{opt2}, we have
\bea
\nono&&\nn g_2^{(s)},U_s^T(x_1^+-x_2^+)\mm\\
&&\geq\f{1}{\aaa}\nn\nabla\phi_s(U_s^Tx)-\nabla\phi_s(U_s^Tx_2^+),U_s^Tx_1^+-U_s^Tx_2^+\mm+\chi_s(U_s^Tx_2^+)-\chi_s(U_s^Tx_1^+). \label{opt02}
\eea
Sum up \eqref{opt01} and \eqref{opt02} and note the 1-strongly convexity of $\phi_s$, it follows that
\be
\f{1}{\aaa}\|U_s^Tx_2^+-U_s^Tx_1^+\|_s^2\leq\nn g_1^{(s)}-g_2^{(s)},U_s^T(x_2^+-x_1^+)\mm\leq\|g_1^{(s)}-g_2^{(s)}\|_s\cdot\|U_s^Tx_2^+-U_s^Tx_1^+\|_s,
\ee
which implies Lemma \ref{lemmap}.
\end{proof}

\noi \textbf{Proof of Lemma \ref{lemma2}:}
\begin{proof}
Follow the definition of $P_s(x,g,\aaa)$ in \eqref{PP},
\bea
\nono&&\|\p_s(x,g_1,\aaa)-\p_s(x,g_2,\aaa)\|_s=\|\f{1}{\aaa}(x-U_s^Tx_1^+)-\f{1}{\aaa}(x-U_s^Tx_2^+)\|_s\\
&&\ \ \ \ \ \ \ \ \ \ \ \ \ \ \ \ \ \ \ \ \ \ \ \ \ \ \ \ \ \ \ \ \ \ \ \ \ \ \ =\f{1}{\aaa}\|U_s^Tx_1^+-U_s^Tx_2^+\|_s,
\eea
which concludes the proof after using Lemma \ref{lemmap}.
\end{proof}

\end{document}